%% file: skeinheisenberg.tex
\documentclass{amsart}
\usepackage{amssymb}
\usepackage{amscd}
\usepackage{amsmath}
\usepackage{amsthm}
\usepackage[dvips]{graphicx}
\usepackage[dvipdfm,%
bookmarks=true,%
bookmarksnumbered=true,%
]{hyperref}

\newcommand{\Z}{{\mathbb Z}}
\newcommand{\R}{{\mathbb R}}

\newcommand{\E}{{\mathcal E}}

\newcommand{\lmd}{\Lambda^*/2\Lambda_0^*}
\newcommand{\Lmd}{\Lambda_0/2\Lambda}

\theoremstyle{plain}
\newtheorem{Thm}{Theorem}[section]
\newtheorem{Lem}[Thm]{Lemma}
\newtheorem{Prop}[Thm]{Proposition}

\theoremstyle{definition}
\newtheorem{Def}[Thm]{Definition}
\newtheorem{Ex}[Thm]{Example}
\newtheorem{Rem}{Remark}[section]

\title[Heisenberg action 
and external edge condition]
{Heisenberg action in skein theory \\ 
and external edge condition}

\author{Hajime Fujita}
\date{}



\address{Department of Mathematics, 
Gakushuin University
1-5-1, Mejiro, Toshimaku, Tokyo
171-8588, Japan} 

\email{hajime@math.gakushuin.ac.jp}

\begin{document}
\maketitle
\begin{abstract}
In this article we give an explicit description of the 
representation matrix of a Heisenberg type action 
constructed 
by Blanchet, Habegger, Masbaum and Vogel. 
%Our construction is based on combinatorial 
%part of Yoshida's construction in \cite{TY}. 
We give the matrix 
in terms of a ribbon graph and its admissible colorings. 
We show that components of the representation matrix 
satisfies the {\it external edge condition}, 
which is a natural combinatorial/geometric condition 
for maps from the first homology of the graph. 
We give the explicit formula of the trace of the action 
in the case of surfaces with colored structure 
using the external edge condition, 
the Verlinde formula and elementary counting arguments. 
Our formula is a generalization of 
the results for a surface without colored structure, 
which are already known. 
\end{abstract}

%\tableofcontents

%%%%%%%%%%%%%\section.1%%%%%%%%%%%%%
\section{Introduction}

In \cite{BHMV}, Blanchet, Habegger, Masbaum and Vogel 
constructed a family of 
topological quantum field theories (TQFT) parameterized by 
non-negative integers $p$ 
as a functor from a two dimensional cobordism category
\footnote{The domain category considered in \cite{BHMV} 
is really that of surfaces with $p_1$-{\it structure}. 
In this article we do not need any $p_1$-structure 
and we omit it for simplicity.} 
to the 
category of modules.
They constructed a TQFT-module $V_p(C)=V_p(C,l,j')$ for 
each {\it surface with colored structure} $(C,l,j')$. 
A surface with colored structure is a triple $(C,l,j')$ consisting 
of a closed oriented surface $C$, a banded link 
$l=l_1\cup\cdots\cup l_n$ in $C$ (i.e. embedded disjoint intervals) 
and a set of fixed colorings $j'=(j'_1,\cdots,j_n')$ 
of components $l_1,\cdots,l_n$. 
Here a coloring is an integer. 
On the other hand, for each {\it admissible coloring} 
of a ribbon graph (Definition~\ref{QCG}), 
one can associate an element of the TQFT-module. 
They also showed that such elements form a free basis of their TQFT-module. 
Our interest in this article is the Heisenberg type action 
on the TQFT-module for $p=2k+4$. 
They defined involutions on the TQFT-module $V_{2k+4}(C)$ 
associated with simple closed curves on the surface, and  
these involutions form a natural action of 
a Heisenberg type group defined as a 
central extension of $H_1(C-l;\Z/2)$. 
Such a Heisenberg action is a starting point of a 
refinement of their TQFT, which is called 
{\it spin-refined TQFT} (\cite{B,BM}).

The purpose of this article is 
to have an explicit representation matrix of the Heisenberg action 
on the TQFT-module in terms of a given 
ribbon graph and its $k$-admissible colorings 
based on combinatorial part of Yoshida's construction in \cite{TY},
and we carry out the computation of the trace of the action 
for a surface with colored structure 
using such an explicit description. 
The main ingredient for the computation of the trace is 
the {\it external edge condition}, 
which is a quite natural geometric/combinatorial condition 
for a map from the homology of the graph to the coefficient ring 
of the TQFT-module. 
The explicit description of the representation matrix 
tells us that the map appeared in the matrix 
satisfies the external edge condition.  
As a corollary of the computation of the trace 
we have the dimension formula 
for the {\it brick decomposition} of the TQFT-module 
for a surface with colored structure. 
The formula for a surface without banded link 
are obtained in \cite{AndMas} and \cite{BHMV} 
using algebro-geometric and skein theoretical technique respectively. 
Our explicit computation works uniformly 
for a surface with non-empty link. 
The computation of the trace also tells us 
that the external edge condition is a characterization 
of the Heisenberg action on the TQFT-module (Remark~\ref{characterization}).

We should note that 
Masbaum gave the computation of the representation matrix 
a long time ago but he never published it. 
Such a computation is a 
special case of the algorithm called {\it TQFT} 
that implemented by 
Masbaum and A'Campo in PARI program to compute representation 
matrices of the mapping class group on the TQFT-module. 
It is possible to compute
representation matrices of the mapping class group 
using similar methods. 
See \cite{A'Campo} for the program TQFT and its applications. 

This article is organized as follows. 
In Section~2 we recall some results obtained in \cite{MV}, 
which we use in Section~3 to compute the representation matrix explicitly. 
In Section~3 we first recall the definition 
of the involutions on the TQFT-modules. 
We give explicit descriptions of 
involutions on the TQFT-module associated with 
a meridian cycle and a longitude cycle. 
In Subsection~3.1 we introduce the external edge condition
(Definition~\ref{defextedgecond}). 
By making use of 
the explicit description one can check that the map appeared in the matrix 
satisfies the external edge condition 
(Proposition~\ref{skeinextedge}). 
In Section~4 we define four lattices associated with the graph 
and realize the first homology group of the surface 
in terms of these lattices (Proposition~\ref{H1}). 
These lattices are introduced in \cite{TY} 
for a closed surface with a pants decomposition. 
In Section~5 we first recall the definition of the Heisenberg type group 
in \cite{BHMV}, and 
the explicit description obtained in Section~3 enable us to describe 
the action of the Heisenberg type group explicitly in terms of the graph 
and its $k$-admissible colorings. 
In Section~6 we compute the trace of the involution 
for a surface with colored structure (Theorem~\ref{trace}). 
To show the formula 
we use the external edge condition, the Verlinde formula 
and elementary counting arguments. 
In Section~7 we demonstrate 
the computation of the dimension formulas 
for a brick decomposition of the TQFT-module 
for a surface with colored structure. 

%\medskip
%\newpage

{\it Convention.} 
In the figures, we use the convention 
that a line represents a ribbon parallel to the plane 
and each trivalent vertex is ordered 
in the counterclockwise direction. 

\medskip

\begin{figure}[htb]
\centering
\input{pic25.tex}
\caption{Convention of orientation}
\end{figure}
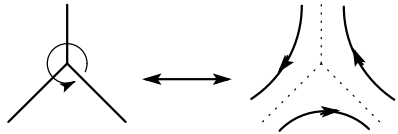
%%%%%%%%%%%%%%%%section%%%%%%%%%%%%%%
\section{Some results in the Temperley-Lieb algebra} 

In this section we recall some results obtained 
in \cite{MV}, which are used in the next section to 
obtain the explicit description of 
involutions on the TQFT-module. 
Fix a non-negative integer $k$. 
We first recall the definition of 
$k$-{\it admissible colorings} of a graph. 

\begin{Def}\label{QCG}
Let $\Gamma$ be a trivalent graph. 
A map $j$ from the set of edges of $\Gamma$ to 
the set of integers $\{0,1,\cdots,k\}$ 
is called a $k$-admissible coloring of $\Gamma$ if 
the following three conditions hold at each vertex with 
edges $f_1$, $f_2$ and $f_3$; 

$$
\left\{\begin{array}{lll} 
j(f_1)+j(f_2)+j(f_3)\in 2\Z \\ 
|j(f_1)-j(f_2)|\leq j(f_3)\leq j(f_1)+j(f_2) \\ 
j(f_1)+j(f_2)+j(f_3)\leq 2k. 
\end{array}\right. 
$$
For a vertex with 
two edges $f_{1}$ and $f_{2}=f_{3}$, 
we interpret these conditions as a
corresponding condition with $j(f_2)=j(f_3)$. 
These three conditions are called the 
{\it quantum Clebsh-Gordon condition of level $k$}. 

For a unitrivalent graph $\Gamma$ with $n$ univalent vertices, 
we fix an $n$-tuple of colorings $j'=(j_1',\cdots,j_n')$ 
(boundary coloring). 
A coloring $j$ of $\Gamma$ is $k$-admissible if it satisfies the 
quantum Clebsh-Gordon condition of level $k$ at each trivalent vertex and 
it is compatible with $j'$, 
i.e, $j(f_l)=j_l'$ for all $f_l$ with a univalent vertex. 
\end{Def}

The {\it Jones-Kauffman skein module} 
of a compact 3-dimensional manifold $M$ over a commutative ring $R$ 
is defined as the $R$-module generated by 
the set of isotopy classes of banded links in $M$ 
meeting $\partial M$ transversally in the banded link in $\partial M$ 
quotiented by the Kauffman relations. 
A colored ribbon graph with a $k$-admissible coloring 
in a compact 3-dimensional manifold 
gives rise to an element in the Jones-Kauffman skein module 
by expanding the graph at each vertex and inserting 
the Jones-Wenzl idempotent at each edge. 
See \cite{BHMV} for detail. 
The coefficient ring of TQFT-module in \cite{BHMV} is defined as a 
kind of cyclotomic ring of degree $2(2k+4)$. 
It contains an indeterminate $A$ 
which should be a $2(2k+4)$-th root of unity. 
Recall the following notations. 
\begin{itemize}
\item
$\displaystyle[n]:=\frac{A^{2n}-A^{-2n}}{A^2-A^{-2}}$. 

\medskip

\item
$[n]!:=[1][2]\cdots[n], \ ([0]!:=1)$.

\medskip

\item
$\langle n \rangle:=(-1)^n[n+1]$.
\end{itemize}
It is known that $[n]$ is invertible for $n=1,2,\cdots, k+1$ and 
$[k+2]=0$. 

Now we recall the results in \cite{MV}, 
which we will use in the next section. 
Those are equalities in the {\it Temperley-Lieb algebra}. 
The Temperley-Lieb algebra is 
the Jones-Kauffman skein module 
of the box $D^2\times [0,1]$ 
with the standard link in $D^2\times \{0\}$ and 
$D^2\times \{1\}$. 
The product is given by the standard product of tangles. 

\medskip

\noindent
{\bf Theorem~1.} [Fusion coefficient]
{\it Put
$$
\langle a,b,c\rangle=(-1)^{i+j+k}
\frac{[i+j+k+1]![i]![j]![k]!}{[i+j]![j+k]![k+i]!},
$$where $i, j$ and $k$ (internal colors) 
are defined by  
\begin{eqnarray*}
i&=&\frac{1}{2}(-a+b+c)\\
j&=&\frac{1}{2}(a-b+c)\\
k&=&\frac{1}{2}(a+b-c). 
\end{eqnarray*} 
Then one has 
the following two equalities in the 
Temperley-Lieb algebra;

\begin{center}
\input{pic2.tex}.
\end{center}

\medskip

\begin{center}
\input{pic11.tex}.
\end{center}
}

\noindent
{\bf Theorem~2.} [Tetrahedron coefficient]
{\it Put 
$$
\left<\begin{array}{ccc}
a & b &c\\ 
d & e & f
\end{array}\right>=
\frac{\prod_{i=1}^3\prod_{j=1}^4[b_i-a_j]!}
{[a]![b]![c]![d]![e]![f]!}
\left(\begin{array}{cccc}
a_1 & a_2 & a_3 & a_4 \\
b_1 & b_2 & b_3 
\end{array}
\right), 
$$where 
$$
\left(\begin{array}{cccc}
a_1 & a_2 & a_3 & a_4 \\
b_1 & b_2 & b_3 
\end{array}
\right)=
\sum_{\max(a_j)\leq \zeta\leq\min(b_i)}
\frac{(-1)^{\zeta}[\zeta+1]!}{\prod_{i=1}^3[b_i-\zeta]!\prod_{j=1}^4[\zeta-a_j]!}
$$for 
\begin{eqnarray*}
a_1&=&\frac{1}{2}(a+b+c), \quad b_1=\frac{1}{2}(b+c+e+f) \\
a_2&=&\frac{1}{2}(b+d+f), \quad b_2=\frac{1}{2}(a+b+d+e) \\
a_3&=&\frac{1}{2}(c+d+e), \quad b_3=\frac{1}{2}(a+c+d+f) \\
a_4&=&\frac{1}{2}(a+e+f). 
\end{eqnarray*}
Then one has 
the following equality in the 
Temperley-Lieb algebra; 
\begin{center}
\input{pic3.tex}
\end{center}
}

\noindent
{\bf Theorem~3.} [Half twist coefficient]
{\it Put
$$
\delta(c;a,b)=(-1)^kA^{ij-k(i+j+k+2)}, 
$$where $i$, $j$ and $k$ are as in Theorem~1. 
Then one has 
the following equality in the 
Temperley-Lieb algebra; 
\begin{center}
\input{pic4.tex}
\end{center}
}

%%%%%%%%%%%%%%%%%section.3%%%%%%%%%%
\section{Computation of representation matrices of the involutions}
In this section we give graphical computations 
for involutions on the TQFT-module 
and obtain the explicit description of them. 
The involution in \cite{BHMV} is defined as follows. 
For each simple closed curve $c$ on a oriented surface $C$, 
set $Z(c)_0:=(C\times [0,1], c_{1/2}(k))$, 
where $c_{1/2}(k)$ is a knot $c\times\{1/2\}$ in 
$C\times [0,1]$ colored by $k$. 
Then $Z(c)_0$ is a cobordism from $C$ to itself 
and it induces an endomorphism on the TQFT-module $V_{2k+4}(C)$. 
The involution in \cite{BHMV} is defined by $Z(c):=(-1)^kZ(c)_0$. 

Now we fix a ribbon unitrivalent graph $\Gamma$. 
Let $B_{\Gamma}$ be 
the product of the interval $[0,1]$ and the ribbon associated with $\Gamma$, 
which is homeomorphic to a handle body. 
Let $C_{\Gamma}$ be the boundary of $B_{\Gamma}$. 
Note that $C_{\Gamma}$ is a closed oriented surface 
equipped with the pants decomposition 
which is dual to the given graph $\Gamma$, 
and $C_{\Gamma}$ may contain a banded link 
corresponding to univalent vertices of $\Gamma$. 
We also have a copy $\Gamma_0=\{0\}\times\Gamma\subset C_{\Gamma}$ 
of $\Gamma$. 

\begin{figure}[htb]
\centering\input{pic26.tex}
\caption{Ribbon graph and surface with a pants decomposition}
\end{figure}
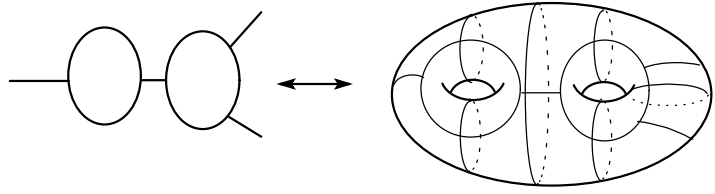

Fix a boundary coloring $j'$ of $\Gamma$. 
For each $k$-admissible coloring $j$ of $\Gamma$, 
a triple $(B_{\Gamma},\Gamma,j)$ defines an element 
of the TQFT-module $V_{2k+4}(C_{\Gamma})=V_{2k+4}(C_{\Gamma},j')$. 
In \cite{BHMV} they showed that the elements 
$\{(B_{\Gamma},\Gamma,j) \ | \ j: k{\rm -admissble}\}$ form a
free basis of $V_{2k+4}(C_{\Gamma})$. 
In this setting 
the involution explained above is described as follows; 
$$
Z(c): (B_{\Gamma},\Gamma,j)\mapsto 
(-1)^k(B_{\Gamma}\cup(C_{\Gamma}\times[0,1]), \Gamma\sqcup c, j\sqcup k), 
$$
where $c$ is a simple closed curve on $C_{\Gamma}$ and 
$j$ is a $k$-admissible coloring of $\Gamma$. 
In this section we have an explicit description of this involution for 
a meridian cycle and a longitude cycle. 
Here a meridian cycle means an element in the kernel of the 
natural map $H_1(C_{\Gamma})\to H_1(B_{\Gamma})$ and a
longitude cycle means an element in $H_1(\Gamma_0)$. 
Before starting the computation 
we note the following equalities. 

\begin{itemize}
\item
$
[k+2-\alpha]=[\alpha]
$

\item
$[k-\alpha]![\alpha+1]!=[k+1]!$

\item
$
\displaystyle{\frac{\langle k-a \rangle}{\langle a\ k \ k-a\rangle}
=(-1)^{a}[a+1]=\langle a \rangle=(-1)^k\langle k-a\rangle}
$

\item
$\displaystyle{
\delta(k-a;k,a)=(-1)^aA^{-a(k+2)}}
$

\item
$\displaystyle{
\frac{\left<\begin{array}{ccc}
k-A & k-B &C\\ 
B & A & k
\end{array}\right>}{\langle k-A \ k-B \ C\rangle}
=(-1)^{\frac{A+B-C}{2}}
\frac{[k-A]![k-B]!\left[\frac{A+B+C}{2}+1\right]!}
{[k+1]!\left[k+1-\frac{A+B-C}{2}\right]!
}}
$
\end{itemize}
All these formulas are shown by the direct computations, 
and the left hand side of 
last three equalities appear in Theorem~1,2 and 3 
for a graph with an edge colored by $k$. 
We also note that a triple $(k,a,c)$ is $k$-admissible 
if and  only if $c=k-a$. 

Hereafter we denote an element $(B_{\Gamma},\Gamma,j)$ 
by $|j\rangle$ for simplicity. 

\vspace{0.5cm}

\noindent
{\bf Meridian.}
%\medskip
%\noindent
Take a meridian circle $\mu$ or a circle around a component of 
the banded link. 
Fix a $k$-admissible coloring $j$ of $\Gamma$. 
Assume that $\mu$ has an edge colored by $a$ as its core. 
See the figure below. 

%\medskip

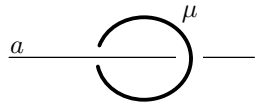
\begin{figure}[htb]
\centering
\input{pic21.tex}
\caption{Meridian $\mu$}
\end{figure}

\noindent
Here the thick line represents 
the cycle colored by $k$, 
and we use such a notation 
throughout this article. 
By the graphical computation below one has that 
$$
Z(\mu):|j\rangle\mapsto (-1)^a|j\rangle. 
$$ 
See also Remark~7.6 in \cite{BHMV}.
In fact one can compute as follows.  

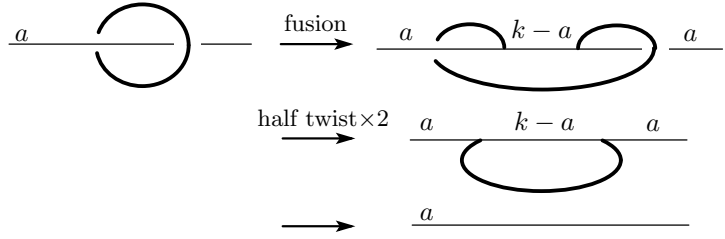
\begin{center}
\input{pic22.tex}
\end{center} 

\noindent
The coefficient associated with these diagrams 
is 
\begin{eqnarray*}
\frac{\langle k-a \rangle}{\langle a\ k \ k-a\rangle}
\delta(k-a;k,a)^2\frac{\langle a\ k \ k-a\rangle}{\langle a \rangle}
&=&A^{-2a(k+2)}\frac{\langle k-a \rangle}{\langle a \rangle}
=(-1)^{k+a}.\\
\end{eqnarray*}

\medskip

\noindent
Here we used the equality 
%\begin{center}
\begin{equation}
\input{pic24.tex}
\label{last}
\end{equation}
%\end{center}

\vspace{0.5cm}

\noindent
{\bf Longitude.} 
%\medskip
%\noindent
Take a cycle $\lambda$ in $\Gamma_0\subset C_{\Gamma}$ 
and a $k$-admissible coloring $j$ of $\Gamma$. 
To compute the explicit description of 
the involution $Z(\lambda)$ on $|j\rangle$ 
it is enough to consider $\lambda${\it -external} 
(resp. {\it internal}) {\it edges}. 

\begin{Def}[$\lambda$-external edge and $\lambda$-internal edge]\label{defextedge}
For a cycle $\lambda$ in $\Gamma$, an edge $f\in \Gamma$ is said to be a 
$\lambda$-{\it external edge} if 
the cycle $\lambda$ on $\Gamma$ does not pass through $f$ 
and one of the vertex of $f$ lies on $\lambda$ and
the other is not. 
If $\lambda$ does not pass through $f$ and 
all vertices of $f$ lie on $\lambda$, then we call 
$f$ is a {\it $\lambda$-internal edge}. 
See Figure~\ref{extint}. 
For given $\lambda$ we denote the set of all 
$\lambda$-external (resp. $\lambda$-internal) edges by 
$\text{Ex}_{\lambda}$ (resp. $\text{In}_{\lambda}$). 
\end{Def}

\begin{figure}[htb]
\centering
\input{pic27.tex}
\caption{External edges and internal edges for $\lambda$
depicted by thick lines.}
\label{extint}
\end{figure}
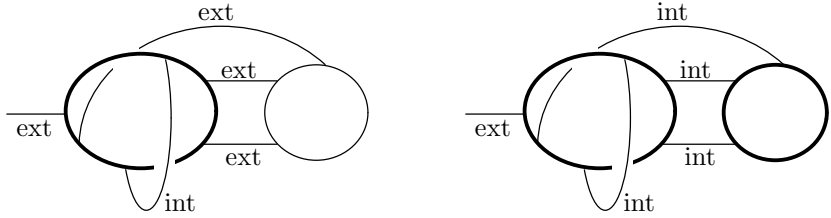

By focusing vertices on (a component of) $\lambda$ and 
cutting $\lambda$-internal edges, one has 
the colored plane diagram 
as a part of $\Gamma\sqcup\lambda$ as in Figure~\ref{D(lambda)}. 
\begin{figure}[htb]
\centering
\input{pic12.tex}
\caption{Diagram around $\lambda$}
\label{D(lambda)}
\end{figure}
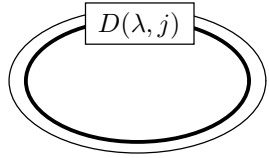
The diagram $D(\lambda,j)$ in the box 
consists of a combination of the two diagrams in Figure~\ref{IandII}. 
(We take a convention for plane diagrams 
in which we draw the cycle $\lambda$ over the graph $\Gamma$.) 
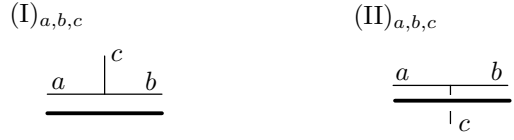
\begin{figure}[htb]
\centering
\input{pic23.tex}
\caption{Diagrams in $D(\lambda , j)$}
\label{IandII}
\end{figure}
Here thick lines in Fig~\ref{IandII} 
represent a part of $\lambda\subset \Gamma_0$ 
colored by $k$ and 
edges colored by $c$ correspond to $\lambda$-external/internal edges. 
We define a number $\epsilon_{\lambda}(v)$ 
for each trivalent vertex $v$ in $D(\lambda,j)$ 
as follows. 
$$
\epsilon_{\lambda}(v)=\left\{\begin{array}{ll}
0 \qquad ({\rm if} \ v \ {\rm is \ type \ (I)}_{a,b,c}) 
\\ a-b \qquad ({\rm if} \ v \ {\rm is \ type \ (II)}_{a,b,c}). \end{array}\right.
$$

To describe the involution $Z(\lambda)$ explicitly, 
let us introduce the action of the homology of the graph  
on the set of $k$-admissible colorings. 

\begin{Def}\label{Gammaaction}
Let $\Gamma$ be a unitrivalent graph. 
For a $k$-admissible coloring $j$ of $\Gamma$ 
and a cycle $\lambda$ on $\Gamma$, 
define $\lambda\cdot j$ by 
$$
\lambda\cdot j:=(\ldots, k-j_l, \ldots), 
$$
that is, change all colors on the edges on $\lambda$ 
from $j_l$ to $k-j_l$, and the others are the same colors as those in $j$.
One can check that $\lambda\cdot j$ is also a 
$k$-admissible coloring and this operation induces 
an action of $H_1(\Gamma;\Z/2)$ on the set of all 
$k$-admissible colorings of $\Gamma$. 
\end{Def}

\begin{Prop}
The involution $Z(\lambda)$ can be described as 
$$
Z(\lambda): |j\rangle \mapsto \delta_{j}(\lambda)|\lambda\cdot j\rangle. 
$$Here $\delta_j$ is defined by 
$$
\delta_j(\lambda):=(-1)^{kn_{\lambda}+\sum j_l'/2}
\prod_{f_l\subset\lambda}A^{\epsilon_{\lambda}(v)(k+2)}
\frac{[k-j(f_l)]!}{[j(f_l)]!}
\frac{\left[\frac{j(f_l)+j(f_{l+1})+j_l'}{2}+1\right]!}
{\left[\frac{k-j(f_l)+k-j(f_{l+1})+j_l'}{2}+1\right]!},
$$where $n_{\lambda}$ is the number of edges on $\lambda$, 
$v$ runs vertices on $\lambda$ and 
$j_l'$ are colorings on $\lambda$-external or internal 
edges. (See Figure~\ref{j's}.)
\end{Prop}

\begin{figure}[htb]
\centering
\input{pic28.tex}
\caption{Colorings of
$\lambda$-external/internal edges}
\label{j's}
\end{figure}
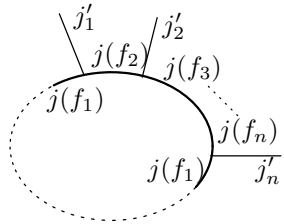

\begin{proof}
Using formulas explained in the previous section
we carry out the 
graphical computations for the diagrams 
(I)$_{a,b,c}$ and (II)$_{a,b,c}$.

\begin{center}
\input{pic6.tex}
\end{center}

The coefficient associated with these diagrams is 
\begin{eqnarray*} 
{\rm (I)}_{a,b,c}&;& 
\frac{\langle k-a \rangle}{\langle a\ k \ k-a\rangle}
\frac{\langle k-b \rangle}{\langle b\ k \ k-b\rangle}
\frac{\left<\begin{array}{ccc}
k-a & k-b &c\\ 
b & a & k
\end{array}\right>}{\langle k-a \ k-b \ c\rangle}\\ 
&=&\langle k-a\rangle\langle k-b\rangle
\frac{\left<\begin{array}{ccc}
k-a & k-b &c\\ 
b & a & k
\end{array}\right>}{\langle k-a \ k-b \ c\rangle}\\
&=&\langle k-a\rangle\langle k-b\rangle
(-1)^{\frac{a+b-c}{2}}
\frac{[k-a]![k-b]!\left[\frac{a+b+c}{2}+1\right]!}
{[k+1]!\left[k+1-\frac{a+b-c}{2}\right]!
}\\
%&=&(-1)^{\frac{a+b+c}{2}}
%\frac{[k-a+1]![k-b+1]!}
%{\left[k+1-\frac{a+b-c}{2}\right]!
%\left[k-\frac{a+b+c}{2}\right]!} \ . 
\end{eqnarray*}

\begin{center}
\input{pic9.tex}
\end{center}

\begin{center}
\input{pic9_.tex}
\end{center}

\vspace{1cm}

The coefficient associated with these diagrams is 
\begin{eqnarray*}
{\rm (II)}_{a,b,c}&;& 
\frac{\langle k-a \rangle}{\langle a\ k \ k-a\rangle}
\cdot\frac{\langle k-b \rangle}{\langle b\ k \ k-b\rangle}
\cdot\frac{\langle k-c \rangle}{\langle c\ k \ k-c\rangle}
\cdot\overline{\delta(k-c;k,c)}\\ 
&& \hspace{5cm}
\cdot\frac{\left<\begin{array}{ccc}
k-a & b &k-c\\ 
c & k & a
\end{array}\right>}{\langle k-a \ b \ k-c\rangle}
\cdot\frac{\left<\begin{array}{ccc}
k-a & k-b &c\\ 
k & k-c & b
\end{array}\right>}{\langle k-a \ k-b \ c\rangle}\\
&=&(-1)^{c}A^{c(k+2)}<k-a><k-b><k-c>\\
&&\hspace{1.5cm}\cdot(-1)^{\frac{a-b+c}{2}}
\frac{[k-a]![k-c]!\left[\frac{a+b+c}{2}+1\right]!}
{[k+1]!\left[k+1-\frac{a-b+c}{2}\right]!}
\cdot (-1)^{\frac{a+b-c}{2}}
\frac{[k-b]![c]!\left[k+1-\frac{a-b+c}{2}\right]!}
{[k+1]!\left[k+1-\frac{a+b-c}{2}\right]!
}\\
&=&A^{(2c+c+a-b+c)(k+2)}
\langle k-a\rangle\langle k-b\rangle
(-1)^{\frac{a+b-c}{2}}
\frac{[k-a]![k-b]!\left[\frac{a+b+c}{2}+1\right]!}
{[k+1]!\left[k+1-\frac{a+b-c}{2}\right]!
}
\\
%&=&(-1)^{k-b}A^{-c(k+2)}
%\frac{[k-a+1]![k-b+1]![k-c+1]![c]!}
%{\left[k+1-\frac{a-b+c}{2}\right]!
%\left[k-\frac{a+b+c}{2}\right]!\left[k+1-\frac{a+b-c}{2}\right]!
%\left[\frac{c}{2}\right]!} \ . 
&=&A^{(a-b)(k+2)}\times {\rm coeff. \ of \ (I)}_{a,b,c}. 
\end{eqnarray*}
Here we note that $A^{2(k+2)}=-1$. 
By using these computations 
together with the equality (\ref{last}) and 
$$
[k-a]!^2=[k-a]!\frac{[k+1]!}{[a+1]!}, 
$$ 
we have the explicit description of $\delta_j(\lambda)$ as in the proposition. 
\end{proof}

\begin{Rem}
The computation given here 
was carried out by Masbaum for a long time ago. 
Moreover he checked that 
by rescaling the basis 
the longitudes act by matrices 
whose entries are $\pm 1$, $\pm\sqrt{-1}$ or zero. Namely, put 
$$
g(j):=\prod_{f:{\rm edges \ of} \ \Gamma}\frac{1}{[j(f)]!}
\prod_{v : {\rm verteces \ of} \ \Gamma}[j(v)+1]! 
$$for a $k$-admissible coloring $j$, where $j(v)$ is defined by 
$$
j(v):=\frac{j(f_1)+j(f_2)+j(f_3)}{2}
$$for a vertex $v$ with edges $f_1$, $f_2$ and $f_3$. 
Then one has 
$$
Z(\lambda): g(j)^{-1}|j\rangle\mapsto 
(-1)^{kn_{\lambda}+\sum j_l'/2}
\prod_{f_l\subset\lambda}A^{\epsilon_{\lambda}(v)(k+2)}
g(\lambda\cdot j)^{-1}|\lambda\cdot j\rangle. 
$$
\end{Rem}

%%%%%%%%%%%%%section.4%%%%%%%%%%
\subsection{External edge condition}

From the description of $Z(\lambda)$ 
obtained in the previous section, 
one can see that the trace of $Z(\lambda)$ is equal to the 
sum of coefficients $\delta_j(\lambda)$ for all $k$-admissible colorings 
$j$ with $\lambda\cdot j=j$. 
For such $\lambda$ and $j$, the diagram $D(\lambda,j)$ consists of 
diagrams (I)$_{k/2,k/2,c}$ and 
(II)$_{k/2,k/2,c}$. 
From now on we assume that $k$ is an even number 
because $k/2$ should be a natural number
\footnote{This fact implies that if $k$ is odd then 
the trace of $Z(\lambda)$ is equal to $0$.}. 
We also note that if $k$ is even and a triple $(k/2,k/2,c)$ is 
admissible then $c$ should be an even number.  
Then one can check the following proposition 
from the description of $\delta_j(\lambda)$. 

\begin{Prop}\label{skeinextedge}
The family of maps $\delta=(\delta_j)$ 
appeared in the representation matrices of $Z(\lambda)$ 
satisfies the {\it external edge condition} defined below. 
\end{Prop}
\noindent

\begin{Def}[External edge condition]\label{defextedgecond}
Let $\alpha=(\alpha_j)$ be a family of maps 
parameterized by the set of $k$-admissible colorings of $\Gamma$ 
compatible with the given boundary coloring, 
where each map $\alpha_j$ is a map from 
$H_1(\Gamma;\Z/2)$ to the coefficient ring $R_k$ of $V_{2k+4}$. 
We say that $\alpha$ satisfies the {\it external edge condition}
if the following condition is satisfied; 
$$
\alpha_j(\lambda)
=(-1)^{\sum_{f_l\in \text{Ex}_{\lambda}}j_l/2} \qquad {\rm if} \ \lambda\cdot j=j.
$$
\end{Def}

We only note that a $\lambda$-internal edge 
can be thought as a $\lambda$-external edge 
with multiplicity $2$, and the coloring on such an edge 
contributes as $(-1)^{\rm even \ number}=1$. 

%%%%%%%%%%%%%%section%%%%%%%%%%%%%%%%%%%%%
\section{Four lattices}
In this section we realize the first homology group 
of the surface in terms of a ribbon graph. 
Let $\Gamma$ be a unitrivalent ribbon graph with 
$3g-3+2n$ edges $E(\Gamma)=\{f_l\}$, $2g-2+n$ trivalent vertices 
and $n$ univalent vertices. 
We decompose the set of edges $E(\Gamma)$ into $E^u(\Gamma)\sqcup E^t(\Gamma)$, 
where $E^u(\Gamma)$ consists of edges with a 
univalent vertex and $E^t(\Gamma)$ consists of edges 
without a univalent vertex. 
Fix $3g-3+2n$ letters $\check{E}(\Gamma)=\{e_l\}$. 
These $e_l$ correspond to simple closed curves 
in the pants decomposition of $C_{\Gamma}^{\circ}$, 
the compact surface obtained by removing 
open disks around the banded link from $C_{\Gamma}$. 
We also have a decomposition 
$\check{E}(\Gamma)=\check{E}^u(\Gamma)\sqcup\check{E}^t(\Gamma)$. 
Here we put $\check{E}^u(\Gamma)({\rm resp.}\check{E}^t(\Gamma))
:=\{e_l \ | \ f_l\in E^u(\Gamma) \ ({\rm resp.} E^t(\Gamma))\}$. 
Now we introduce four lattices $\Lambda_0$, 
$\Lambda$, $\Lambda_0^*$ and $\Lambda^*$ as follows
\footnote{These lattices are introduced in \cite{TY} 
for a closed surface with a pants decomposition.}. 

\begin{Def}
Let $\R\langle\check{E}(\Gamma)\rangle$ be the 
vector space generated by $\check{E}(\Gamma)=\{e_l\}$ over $\R$. 
Let $\Lambda_0:=\Z\langle\check{E}(\Gamma)\rangle$ 
be the standard lattice in $\R\langle\check{E}(\Gamma)\rangle$ and 
$\Lambda$ be the lattice generated by the elements 
$\{E_1^i,E_2^i, E_3^i \ | \ i=1,\cdots,2g-2+n \}$, 
where these elements are defined  as follows; 
For each trivalent vertex $v_i$ with edges $f_{i_1}$, $f_{i_2}$ and $f_{i_3}$ 
define $E_1^i$, $E_2^i$ and $E_3^i$ by 
\begin{eqnarray*}
E_1^i&=&\frac{1}{2}\left(-e_{i_1}+e_{i_2}+e_{i_3}\right)\\ 
E_2^i&=&\frac{1}{2}\left(e_{i_1}-e_{i_2}+e_{i_3}\right)\\
E_3^i&=&\frac{1}{2}\left(e_{i_1}+e_{i_2}-e_{i_3}\right).\\
\end{eqnarray*} 
If a vertex 
$v_i$ has only two edges $f_{i_1}$ and $f_{i_2}=f_{i_3}$ 
then we interpret as 
\begin{eqnarray*}
E^i_1&:=&-\frac{1}{2}e_{i_1}+e_{i_2}\\
E^i_2&=&E^i_3:=\frac{1}{2}e_{i_1}. 
\end{eqnarray*}

Let $\R\langle E^t(\Gamma)\rangle$ be the $\R$-vector space generated by 
$E^t(\Gamma)$. 
Let $\Lambda_0^*=\Z\langle E^t(\Gamma)\rangle$ be the standard lattice 
in $\R\langle E^t(\Gamma)\rangle$ and 
$\Lambda^*$ be the sublattice of $\Lambda_0^*$ defined by 
$$
\Lambda^*:=\left\{
\sum_ln_lf_l\in \Lambda_0^* \ | \ n_{i_1}+n_{i_2}+n_{i_3}\in 2\Z\right\}. 
$$
There is a natural pairing between $\R\langle\check{E}(\Gamma)\rangle$ and 
$\R\langle E^t(\Gamma)\rangle$ denoted by $\cdot$, which is defined by 
$e_l\cdot f_{l'}=\delta_{l,l'}$. 
Note that this pairing is non-degenerate except the subspace 
generated by $\check{E}^u(\Gamma)$. 
\end{Def}

\begin{Rem}
Note that $\Lambda$ is a sublattice of $\Lambda_0$ 
with $\Lambda/\Lambda_0\cong (\Z/2)^{2g-2+2n-n'}$, where $n'=n$ ($n\geq 1$), $1$ ($n=0$). 
This follows from the facts 
$E^i_1\equiv E^i_2\equiv E^i_3 \bmod \Lambda_0$ 
for $i=1,\cdots,2g-2+n$ 
and $\sum_{i}E_1^i\equiv 0\bmod \Lambda_0$ if $n=0$.
\end{Rem}

\begin{Prop}\label{H1}
We have the followings; 
\begin{enumerate}
\item There is a canonical isomorphism 
$\Lambda^*/2\Lambda_0^*\cong H_1(\Gamma;\Z/2)(\cong (\Z/2)^g)$. 
\item 
The inclusion $\Gamma_0\subset C_{\Gamma}$ 
induces a decomposition $H_1(C_{\Gamma}^{\circ};\Z/2)\cong\Lmd\oplus\lmd$. 
\end{enumerate}
\end{Prop}

\begin{proof}
1. The required isomorphism is given as follows. 
By definition each $\lambda\in\lmd$ can be represented by an  
element in $\Lambda^*$ which has two (or no) nontrivial 
entries at each trivalent vertex. Putting edges corresponding 
to nontrivial entries in $\lambda$ together we obtain a cycle in $\Gamma$. 
Conversely for a given cycle in $\Gamma$ we have an element in $\lmd$ 
whose nontrivial entries correspond with edges on the cycle. 

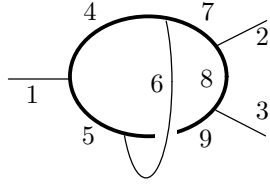
\begin{figure}[htb]
\centering
\input{pic20.tex}
\caption{The thick line corresponds to $\lambda=f_4+f_5+f_7+f_8+f_9$}
\label{H1Gamma}
\end{figure}

\noindent
2. The elements in $2\Lambda$ correspond to 
boundaries of the pants decomposition of the surface $C_{\Gamma}^{\circ}$ 
up to orientation. Hence there is a natural map 
$\Lmd\to H_1(C_{\Gamma}^{\circ};\Z/2)$, 
and one can check that this is an inclusion. 
Then the rank of cokernel of this inclusion is equal to $g$ 
because we have that $\Lmd\cong(\Z/2)^{g+n'-1}$. 
On the other hand we have a natural map 
$\lmd\cong H_1(\Gamma_0;\Z/2)\to H_1(C_{\Gamma}^{\circ};\Z/2)$ 
via the inclusion $\Gamma_0\subset C_{\Gamma}$ . 
It induces an injection $\lmd\to H_1(C_{\Gamma}^{\circ};\Z/2)/(\Lmd)$, 
and this is an isomorphism because these have the same rank $g$. 
This construction gives a splitting of the exact sequence 
$0\to\Lmd\to H_1(C_{\Gamma}^{\circ};\Z/2)\to 
H_1(C_{\Gamma}^{\circ};\Z/2)/(\Lmd)\cong\lmd\to 0$. 
\end{proof}

\begin{Rem}
Note that the subgroup $\Lambda_0^u/2\Lambda^u$ is 
isomorphic to the subgroup of $H_1(C_{\Gamma}^{\circ};\Z/2)$ 
generated by boundaries of $C_{\Gamma}^{\circ}$,
where $\Lambda^u_0$ is the lattice generated by $\check{E}^u(\Gamma)$ 
and we put $2\Lambda^u:=\Lambda_0^u\cap2\Lambda$.  
\end{Rem}

\begin{Rem}\label{sympform}
The isomorphism $H_1(C_{\Gamma}^{\circ};\Z/2)\cong \Lmd\oplus\lmd$ is 
compatible with the intersection pairing. Namely, 
the mod $2$ intersection pairing between meridians and 
longitudes coincides with the natural pairing 
between $\Lmd$ and $\lmd$ induced from the pairing 
between $\Lambda_0$ and $\Lambda_0^*$. 
We denote this pairing by 
$\cdot:\Lmd\oplus\lmd\to \Z/2, \ (\mu,\lambda)\mapsto \mu\cdot\lambda$. 
\end{Rem}

%%%%%%%%%%%%%%%%%%%%%section%%%%%%%%%%%%%%%
\section{Representation matrix of the Heisenberg action}
In this section we describe the Heisenberg action 
in \cite{BHMV} explicitly in terms of a ribbon graph 
and its admissible colorings. 
We first review the construction of the Heisenberg type group in \cite{BHMV}. 
For an oriented surface $C$, 
let $H(C)$ be the direct product of 
$\Z$ and $H_1(C;\Z)$ with multiplication given by 
$(n,x)(m,y):=(n+m+x\cdot y, x+y)$, where $x\cdot y$ is the 
intersection number of $x$ and $y$. 
We denote $(n,0)$ by $u^n$ for $n\in \Z$ 
and $[x]=(0,x)$ for $x\in H_1(C;\Z)$. 

\begin{Def}
Let $\E(C)$ be the quotient of $H(C)$  by the 
subgroup generated by $u^4$ and elements $[2x]$. 
One can check that $\E(C)$ is a central extension 
$$
\Z/4\to \E(C)\to H_1(C;\Z/2). 
$$
\end{Def}

For the surface $C_{\Gamma}$ associated with a ribbon graph $\Gamma$, 
we denote $\E(C_{\Gamma})$ by $\E(\Gamma)$. 
We can define a map 
$\tau:\Lmd\oplus\lmd\to \E(\Gamma)$ as follows. 
We first note that $\mu\in\Lmd$ is represented by 
a disjoint union of simple closed curves $\tilde\mu$ on $C_{\Gamma}$, 
and it determined an element in $H_1(C_{\Gamma};\Z)$ up to sign. 
The induced element $\tau(\mu):=[\tilde\mu]\in\E(\Gamma)$ depends only on $\mu$ 
since $\tilde\mu\cdot\tilde\mu=0$ and hence $[\tilde\mu]=[-\tilde\mu]$. 
Similarly we have an element $\tau(\lambda)\in \E(\Gamma)$. 
We define $\tau(\mu,\lambda):=u^{\mu\circ\lambda}\tau(\lambda)\tau(\mu)$, 
where $\mu\circ\lambda$ is the geometric intersection number 
of $\mu$ and $\lambda$. 

\begin{Rem}
Put $\E'(\Gamma):=\Z/2\times\left(\Lmd\oplus\lmd\right)$ as a set. 
Define a group structure on $\E'(\Gamma)$ by 
$$
(c_1,\mu_1,\lambda_1)\cdot(c_2,\mu,\lambda_2)=
(c_1c_2(-1)^{\lambda_2\cdot\mu_1},\mu_1+\mu_2,\lambda_1+\lambda_2),
$$where $\cdot$ is the natural pairing 
between $\Lmd$ and $\lmd$. 
Then $\E'(\Gamma)$ is a central extension 
$$
\Z/2\to \E'(\Gamma)\to \Lmd\oplus\lmd. 
$$
Then one can check that $\E'(\Gamma)$ is a reduction of 
$\E(\Gamma)$ to a $\Z/2$-extension. 
Namely the $\Z/4$-extension induced from $\E'(\Gamma)$ 
via the natural inclusion $\Z/2\hookrightarrow\Z/4$ 
is isomorphic to $\E(\Gamma)$. 
\end{Rem}

Fix a non-negative integer $k$ and a boundary coloring $j'=(j'_l)$. 
We have the TQFT-module $V_{2k+4}(C_{\Gamma})=V_{2k+4}(C_{\Gamma};j')$,
which has a basis parameterized by 
the finite set of all $k$-admissible colorings. 
By previous computations, 
we have the following descriptions of 
the involutions of meridian cycles and longitude cycles; 

\vspace{0.5cm}

\noindent
{\bf Meridian.} 
%\medskip
%\noindent
For $\mu\in \Lmd$ and a $k$-admissible coloring $j$ 
one has 
$$
\tau(\mu):|j\rangle\mapsto (-1)^{j_{\mu}}|j\rangle, 
$$where 
$j_{\mu}$ is defined by 
$j_{\mu}:=\sum_{\epsilon_l\neq 0}j_l$ for $\mu=\sum_l\epsilon_le_l $. 

\vspace{0.5cm}

\noindent
{\bf Longitude.} 
%\medskip
%\noindent
For $\lambda\in\lmd$ and a $k$-admissible coloring $j$ 
one has 
$$
\tau(\lambda):|j\rangle\mapsto\delta_j(\lambda)|\lambda\cdot j\rangle, 
$$where the coefficient $\delta_j(\lambda)$ can be 
computed as in the previous section and 
it satisfies the {\it external edge condition} 
(Definition~\ref{defextedgecond}). 

\medskip

\noindent
Combining these descriptions 
and the action of central elements (See \cite[Proposition~7.5]{BHMV}.) 
we have the following description 
of $\E(\Gamma)$-action on $V_{2k+4}(C_{\Gamma})$; 
\begin{eqnarray*}
\rho^{(k)}:\E(\Gamma)&\to& GL(V_{2k+4}(C_{\Gamma}))\\ 
u^m\tau(\mu,\lambda)&\mapsto& 
\left(|j\rangle\mapsto 
A^{(k+2)^2(m+\mu\circ\lambda)}
(-1)^{(k+1)(m+\mu\circ\lambda)+j_\mu}\delta_j(\lambda)|\lambda\cdot j\rangle\right).
\end{eqnarray*}

\begin{Rem}
If $k$ is an even number the one has that 
$\rho^{(k)}(\tau(a))\rho^{(k)}(\tau(b))=
(-1)^{\frac{k}{2}a\cdot b}\rho^{(k)}(\tau(a+b))$ 
by direct computation. 
In particular in the case of even $k$, 
the representation $\rho^{(k)}$ is an abelian. 
\end{Rem}

%%%%%%%%section.%%%%%%%%%%%%%%%%%%%%%%%%%%%
\section{Trace of the Heisenberg action }

In this section we compute the trace of the 
involutions on the TQFT-module by making use of 
the previous results. 
The formula for closed surfaces with empty link
is obtained in \cite{AndMas} and \cite{BHMV}. 
Our computation works uniformly for 
surfaces with non-empty link,  and 
the formula is the same as the formula 
for a surface without banded link up to 
some factor ($\in \{-1,0,1\}$) which is 
determined by the boundary coloring. 

%%\newpage 

Let $\Gamma$ be a ribbon unitrivalent graph. 
Fix a boundary coloring $j'$. 
For a non-negative integer $k$ 
we denote by $QCG_k(\Gamma;j')$ 
the set of all $k$-admissible coloring of $\Gamma$ 
compatible with the given boundary coloring $j'$. 
By the description of $\E(\Gamma)$-action, 
the trace of $\rho^{(k)}(\tau(\mu,\lambda))$ is 
given by 
$$
Tr(\rho^{(k)}(\tau(\mu,\lambda));V_{2k+4}(C_{\Gamma};j'))=
\sum_{j\in QCG_k^{\lambda}(\Gamma; j')}
A^{(k+2)^2{\mu\circ\lambda}}(-1)^{(k+1)\mu\circ\lambda+j_{\mu}}\delta_j(\lambda),
$$where $QCG_k^{\lambda}(\Gamma; j')$ is 
the set of all $k$-admissible colorings 
fixed by the action of $\lambda$ : 
%the fixed set with respect to  
%$\lambda$-action; 
$$
QCG_k^{\lambda}(\Gamma; j'):=\{j\in QCG_k(\Gamma;j') \ | \ \lambda\cdot j= j\}.
$$ 
In this section we compute the right hand side directly 
and show the following formula. 

\begin{Thm}\label{trace}
We have the following trace formula; 
$$
Tr(\rho^{(k)}(\tau(\mu,\lambda));V_{2k+4}(C_{\Gamma};j'))=
\gamma(j')
\frac{1+(-1)^k}{2}\left(\frac{k+2}{2}\right)^{g-1}
$$for $\mu\in\Lmd$ and $\lambda(\neq 0)\in\lmd$ or 
for $\mu\notin \Lambda^u_0/2\Lambda^u$ and $\lambda=0$, 
where $\gamma(j')\in\{-1,0,1\}$ is defined below (Definitions/Notations~(3)). 
For $\mu=\sum\epsilon_le_l\in \Lambda^u_0/2\Lambda^u$, 
one has 
$Tr(\rho^{(k)}(\mu,0);V_{2k+4}(C_{\Gamma};j'))=
(-1)^{\sum\epsilon_lj_l'}{^{\#}QCG_k(\Gamma;j')}$. 
\end{Thm}

The last part follows from the description 
of the action of a cycle around a component of the banded link. 
The case for odd $k$ follows from the following lemma.  
\begin{Lem}\label{oddtrace}
If $k$ is an odd number, then 
we have $Tr(\rho^{(k)}(\tau(\mu,\lambda));V_{2k+4}(C_{\Gamma};j'))=0$ 
for $\mu\in\Lmd$ and $\lambda(\neq 0)\in\lmd$ or 
$\mu\notin\Lambda_0^u/2\Lambda^u$ and $\lambda=0$. 
\end{Lem}

\begin{proof}
First note that 
if $\lambda\cdot j=j$ 
then $j_l=k/2$ 
for all $l$ with $\lambda_l\neq 0$. 
But $k/2$ is not an integer if $k$ is odd. 
This implies that for an odd number $k$ and $\lambda\neq 0$ 
we have $QCG_k^{\lambda}(\Gamma;j')=\emptyset$, and hence 
$Tr(\rho^{(k)}(\mu,\lambda);V_{2k+4}(C_{\Gamma};j'))=0$. 
Now we assume that $k$ is an odd number and $\lambda=0$, 
$\mu\notin\Lambda^u_0/2\Lambda^u$. 
In this case we have 
$$
Tr(\rho^{(k)}(\mu,0);V_{2k+4}(C_{\Gamma};j'))
=\sum_{j\in QCG_k^{\lambda}(\Gamma; j')}(-1)^{j_{\mu}}
=^{\#}\{j\in QCG_{k} \ | \ j_{\mu}\in 2\Z \ \}
-^{\#}\{j\in QCG_{k} \ | \ j_{\mu}\notin 2\Z \ \}. 
$$
Take and fix a cycle $\lambda_{\mu}\in\lmd$ on $\Gamma$ 
with $\mu\cdot\lambda_{\mu}\equiv 1\bmod 2$. 
Then a map 
\begin{eqnarray*}
\{j\in QCG_k \ | \ j_{\mu}\in 2\Z\}&\longleftrightarrow&
\{j\in QCG_k \ | \ j_{\mu}\notin 2\Z\}\\
j&\longleftrightarrow&\lambda_{\mu}\cdot j
\end{eqnarray*}
gives a bijection if $k$ is an odd number because of the relation 
$$
(\lambda_{\mu}\cdot j)_{\mu}=
k\lambda_{\mu}\cdot\mu+j_{\mu} \bmod \Z, 
$$
in particular we obtain 
$Tr(\rho^{(k)}(\mu,0);V_{2k+4}(C_{\Gamma};j'))=0$. 
\end{proof}

Now we make several preparations to show the theorem for even $k$. 
Here after we assume that $k$ is an even natural number. 
We introduce several notations. 
(See also Example~\ref{notations}.)

\vspace{0.5cm}

\noindent
{\bf Definitions/Notations.}
\begin{enumerate}
\item
For given $(\mu,\lambda)\in \Lmd\oplus\lmd$ 
take and fix representatives $\mu=\sum_{l}\epsilon_le_l$  
and $\lambda=\sum_{l}\epsilon'_lf_l$ (we denote them by the same letters), 
and we define $\mu(\lambda)$ by 
$\mu(\lambda):=\sum_{l}\epsilon_l(1-\epsilon_l')e_l$. 
We consider the 
decomposition of the graph $\Gamma$ into 
$\Gamma(\lambda)\cup\Gamma'(\lambda)$ and $\Gamma(\lambda)\cup\Gamma'(\lambda)_{\mu}$
as follows. 

\begin{eqnarray*}
&&\Gamma(\lambda):=\{f_l,v_i,w_i \ | \ f_l \ {\rm with} \ \lambda_l\neq 0 \ 
{\rm or} \ f_l\in \text{Ex}_{\lambda}\cup \text{In}_{\lambda},  \ {\rm and} \ 
v_i \ {\rm or} \ w_i  \ {\rm is \ a \ vertex \ of} \ f_l\},\\
&&\Gamma'(\lambda):=\{f_l,v_i,w_i \ | \ f_l \ {\rm with} \ \lambda_l= 0 \ 
{\rm or} \ f_l\in \text{Ex}_{\lambda},  \ {\rm and} \
v_i \ {\rm or} \ w_i  \ {\rm is \ a \ vertex \ of} \ f_l\}.\\
&&\Gamma'(\lambda)_{\mu} \ ; \ {\rm the \ graph \ 
obtained \ by \ cutting \ the \ edges \ in} \ 
\Gamma'(\lambda) \ {\rm with} \ \mu(\lambda)_l\neq 0.
\end{eqnarray*}

\item
For a given graph $\cdot$,  we denote by $E^u(\cdot)$ the set of edges which have 
one univalent vertex and by $E^t(\cdot)$ the set of edges 
which have no univalent vertex. 
Note that $E^u(\Gamma(\lambda))=\text{Ex}_{\lambda}$. 
We define the subset of edges as follows; 

\begin{eqnarray*}
&& E^u_{\lambda}:=\text{Ex}_{\lambda}\cap E^u(\Gamma)\subset E^u(\Gamma(\lambda))\\ 
&& E^t_{\lambda}:=\text{Ex}_{\lambda}\cap E^t(\Gamma)\subset E^u(\Gamma(\lambda))\\
&& \overline{E^u_{\lambda}}:=E^u(\Gamma)- E^u_{\lambda}
\subset E^u(\Gamma'(\lambda))\\
&& E^u_{\mu}:=E^u(\Gamma'(\lambda)_{\mu})
-\left(\text{Ex}_{\lambda}\cup\overline{E^u_{\lambda}}\right)
\subset E^u(\Gamma'(\lambda)_{\mu}). 
\end{eqnarray*}

\item
For a subset $E'$ of $E^u(\cdot)$ for a given graph $\cdot$, 
we denote by $j'(E')$ the set of given boundary colorings of $E'$, 
and we put $\sum j'(E'):=\sum_{f_l\in E'}j'_l$. 
We also define $\gamma(j'(E'))\in \{-1,0,1\}$ by 
$$
\gamma(j'(E')):=\left\{\begin{array}{ll}
\displaystyle{(-1)^{\sum j'(E')/2} \quad (\forall f_l\in E',\ j_l'\in 2\Z)} \\ 
0 \quad (\exists f_l\in E', \ j_l'\notin 2\Z). 
\end{array}\right.
$$

\item
We put $n_1:={^\#}E_{\lambda}^u$, 
$n_2:={^\#}\overline{E_{\lambda}^u}$ 
and we assume that $\Gamma(\lambda)$ has $3g_1-3+2(n_1+m)$ edges and 
$\Gamma'(\lambda)$ has $3g_2-3+2(n_2+m)$ edges 
(then $g=g_1+g_2+m-1$ and $m={^\#}E^t_{\lambda}$).
Note that $E_{\mu}^u$ has an even element, 
and we put $2m':={^\#}E^u_{\mu}$. 

\item
Let $QCG_k^{\lambda}(\Gamma;j'(\Gamma))$ be the subset of 
all admissible colorings for $(\Gamma;j'(\Gamma))$ fixed by 
the action of $\lambda$, i.e, 
$$
QCG_k^{\lambda}(\Gamma;j'(\Gamma))=
\{j\in QCG_k(\Gamma;j'(\Gamma)) \ | \ j_l=k/2 \ {\rm if} \ \lambda_l\neq 0\}. 
$$
Similarly we can define the subset 
$QCG_k^{\lambda}(\Gamma(\lambda);j'(\Gamma(\lambda)))
\subset QCG_k(\Gamma(\lambda);j'(\Gamma(\lambda)))$. 
%$:=\{j\in QCG_k(\Gamma(\lambda);Ex(\lambda)) \ | \ 
%\lambda\cdot j=j\}
%=\{j\in QCG_k(\Gamma(\lambda);Ex(\lambda)) \ | \ 
%j_l=\frac{k}{2} \ {\rm if} \ \lambda_l\neq 0\}.$
\end{enumerate}

\vspace{1cm}

\begin{Ex}\label{notations}
An example of $\Gamma(\lambda)$, 
$\Gamma'(\lambda)$ and $\Gamma'(\lambda)_{\mu}$ is shown 
in Figure~\ref{exampledef}. 
We take $\mu=e_6+e_8+e_{15}$ and $\lambda=f_4+f_5+f_7+f_8+f_9$. 
In this case we have $\mu(\lambda)=e_{15}$, 
$E^u_{\lambda}=\{f_1\}$, $E^t_{\lambda}=\{f_{10},f_{11}\}$, 
$\overline{E^u_{\lambda}}=\{f_2,f_3\}$ and 
$E^u_{\mu}=\{f_{15},f_{15}\}$. 
%\newpage

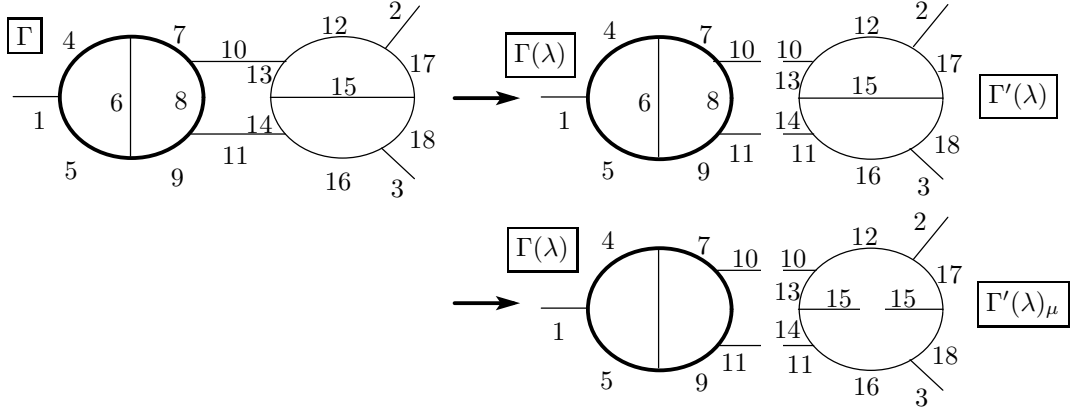
\begin{figure}[htb]
\begin{center}
\input{pic19.tex}
\end{center}
\caption{Examples of $\Gamma(\lambda)$, $\Gamma'(\lambda)$ and $\Gamma'(\lambda)_{\mu}$}
\label{exampledef}
\end{figure}
\end{Ex}
%\noindent

%\vspace{1cm}
%\newpage

Theorem~\ref{trace} is shown by the following two propositions. 

%\newpage

\begin{Prop} \label{main1}
We have the equality \ 
$^{\#}QCG_k^{\lambda}(\Gamma(\lambda);j'(\Gamma(\lambda)))
=\gamma(j'(\text{Ex}_{\lambda}))^2
\left(\frac{k+2}{2}\right)^{g_1-1}$.
%\prod_{f_l\in \text{Ex}_{\lambda}}\sin^2\frac{j'_l+1}{2}\pi$
\end{Prop}

\begin{Prop}\label{main2}
We have the equality 
$$
\sum_{j'(E^t_{\lambda}), \ j'(E^u_{\mu})}(-1)^{\sum j'(E^t_{\lambda})/2
+\sum j'(E^u_{\mu})}
{^{\#}QCG_k(\Gamma'(\lambda)_{\mu};j'(\Gamma'(\lambda)_{\mu}))}
=\gamma(j'(\overline{E^u_{\lambda}}))
\left(\frac{k+2}{2}\right)^{g_2-1+m}.
$$
\end{Prop}

By using these propositions 
we can prove Theorem~\ref{trace} 
by an elementary counting argument. 
\begin{proof}[Proof of Theorem~\ref{trace}]
Note that 
for an even number $k$ one has 
$A^{(k+2)^2{\mu\circ\lambda}}(-1)^{(k+1)\mu\circ\lambda+j_{\mu}}=
(-1)^{\frac{k}{2}\mu\cdot\lambda+j_{\mu}}$. 
By using this equality and the fact that $\delta_j$ satisfies the 
external edge condition, 
the trace can be computed in the following way; 
\begin{eqnarray*}
&&\sum_{j\in QCG_k^{\lambda}(\Gamma;j'(\Gamma))}
(-1)^{\frac{k}{2}\mu\cdot\lambda+j_{\mu}}\delta_j(\lambda)=
\sum_{j\in QCG_k^{\lambda}(\Gamma;j'(\Gamma))}
(-1)^{\frac{k}{2}\mu\cdot\lambda+j_{\mu}+\sum_{f_l\in \text{Ex}_{\lambda}}j_l/2}\\
&=&(-1)^{\frac{k}{2}\mu\cdot\lambda}\left(
^{\#}\{j\in QCG_{k}^{\lambda} \ | \ j_{\mu}+
\sum_{f_l\in \text{Ex}_{\lambda}}j_l/2\equiv 0\bmod 2\}
-^{\#}\{j\in QCG_{k}^{\lambda} \ | \ j_{\mu}
+\sum_{f_l\in \text{Ex}_{\lambda}}j_l/2\equiv 1\bmod 2\}\right)\\
&=&(-1)^{\frac{k}{2}\mu\cdot\lambda}\left(
^{\#}\{j\in QCG_{k}^{\lambda} \ | \ j_{\mu(\lambda)}+
\frac{k}{2}\mu\cdot\lambda+\sum_{f_l\in \text{Ex}_{\lambda}}j_l/2\equiv 0\bmod 2\}\right.
\\&&\hspace{5cm}
\left.-^{\#}\{j\in QCG_{k}^{\lambda} \ | \ j_{\mu(\lambda)}+
\frac{k}{2}\mu\cdot\lambda+\sum_{f_l\in \text{Ex}_{\lambda}}j_l/2\equiv 1\bmod 2\}\right)\\
&=&
{^{\#}}\{j\in QCG_{k}^{\lambda} \ | \ j_{\mu(\lambda)}+
\sum_{f_l\in \text{Ex}_{\lambda}}j_l/2 \equiv 0\bmod 2\}
-^{\#}\{j\in QCG_{k}^{\lambda} \ | \ j_{\mu(\lambda)}
+\sum_{f_l\in \text{Ex}_{\lambda}}j_l/2 \equiv 1\bmod 2\}\\
&=&
{^{\#}}\{j\in QCG_{k}^{\lambda} \ | \ j_{\mu(\lambda)}, 
\sum_{f_l\in \text{Ex}_{\lambda}}j_l/2 \in 2\Z\}
-^{\#}\{j\in QCG_{k}^{\lambda} \ | \ j_{\mu(\lambda)}\in 2\Z, 
\sum_{f_l\in \text{Ex}_{\lambda}}j_l/2 \notin 2\Z \}
\\&&
+{^{\#}}\{j\in QCG_k^{\lambda} \ | \ j_{\mu(\lambda)}\notin 2\Z,
\sum_{f_l\in \text{Ex}_{\lambda}}j_l/2 \notin 2\Z \}
-^{\#}\{j\in QCG_{k}^{\lambda} \ | \ j_{\mu(\lambda)}\notin 2\Z,
\sum_{f_l\in \text{Ex}_{\lambda}}j_l/2 \in 2\Z \}\\
&=&
\sum_{j'(\text{Ex}_{\lambda})}(-1)^{\sum j'(\text{Ex}_{\lambda})/2}
{^{\#}QCG_k^{\lambda}(\Gamma(\lambda);j'(\Gamma(\lambda)))}
{^{\#}}\{j\in QCG_k(\Gamma'(\lambda);j'(\Gamma'(\lambda))) \ | \ 
j_{\mu(\lambda)}\in2\Z \}\\
&&\hspace{1.2cm}-\sum_{j'(\text{Ex}_{\lambda})}(-1)^{\sum j'(\text{Ex}_{\lambda})/2}
{^{\#}QCG_k^{\lambda}(\Gamma(\lambda);j'(\Gamma(\lambda)))}
{^{\#}}\{j\in QCG_k(\Gamma'(\lambda);j'(\Gamma'(\lambda))) \ | \ 
j_{\mu(\lambda)}\notin 2\Z \}\\
\end{eqnarray*}
\begin{eqnarray*}
&=&
\left(\frac{k+2}{2}\right)^{g_1-1}
\prod_{f_l\in E^u_{\lambda}}\sin^{2+1}\frac{j_l'+1}{2}\pi
\sum_{j'(E^t_{\lambda})}(-1)^{\sum j'(E_{\lambda}^t)/2}\\
&&\times\Bigl(
{^{\#}}\{j\in QCG_k(\Gamma'(\lambda);j'(\Gamma'(\lambda))) \ | \ 
j_{\mu(\lambda)}\in 2\Z \}
-{^{\#}}\{j\in QCG_k(\Gamma'(\lambda);j'(\Gamma'(\lambda))) \ | \ 
j_{\mu(\lambda)}\notin 2\Z \}\Bigr)\\
&=&\left(\frac{k+2}{2}\right)^{g_1-1}
\prod_{f_l\in E^u_{\lambda}}\sin\frac{j_l'+1}{2}\pi
\sum_{j'(E^t_{\lambda})}(-1)^{\sum j'(E_{\lambda}^t)/2+\sum j'(E^u_{\mu})}
{^{\#}QCG_k(\Gamma'(\lambda)_{\mu};j'(\Gamma'(\lambda)_{\mu}))}
\\
&=&\left(\frac{k+2}{2}\right)^{g_1-1}
\prod_{f_l\in E^u_{\lambda}}\sin\frac{j_l'+1}{2}\pi\cdot
\left(\frac{k+2}{2}\right)^{g_2-1+m}
\prod_{f_l\in \overline{E^u_{\lambda}}}\sin\frac{j_l'+1}{2}\pi
=\gamma(j')\left(\frac{k+2}{2}\right)^{g-1}.
\end{eqnarray*}
Here we used the factorization property in the 6th 
equal sign, Proposition~\ref{main1} in the 7th equal sign 
and Proposition~\ref{main2} in the 8th equal sign. 
\end{proof}

Now we start the proof of Proposition~\ref{main1} and 
\ref{main2}. 

\begin{proof}[Proof of Proposition~\ref{main1}]
First note that a triple of integers $(k/2,k/2,j_l)$ 
satisfies the $QCG_k$-condition if and only if 
$j_l\in\{0,2,\cdots,k\}\subset 2\Z$, 
and hence one has that 
$QCG_k^{\lambda}(\Gamma(\lambda);j'(\Gamma(\lambda)))\neq\emptyset$ 
only if $j'(\Gamma(\lambda))\subset \{0,2,\cdots,k\}^{n_1+m}\subset 2\Z^{n_1+m}$. 
Next we assume that this condition is satisfied. 
By the construction there are $g_1-1$ edges in 
$\text{In}_{\lambda}\subset \Gamma(\lambda)$. 
On the other hand elements in 
$QCG_k^{\lambda}(\Gamma(\lambda);j'(\lambda)))$ 
have coloring $k/2$ at edges on $\lambda$
and integer colorings $0,\cdots,k$ at edges in $\text{In}_{\lambda}$. 
In addition such integer coloring at edges in $\text{In}_{\lambda}$ 
can be taken arbitrary. 
This implies that there are $\left(\frac{k+2}{2}\right)^{g_1-1}$ 
elements in $QCG_k^{\lambda}(\Gamma(\lambda);j'(\Gamma(\lambda)))$. 
\end{proof}

To show Proposition~\ref{main2}, we use the following 
cerebrated Verlinde formula. 

\medskip

\noindent
{\bf Verlinde formula.} \quad 
 Let $\Gamma$ be a unitrivalent graph and $j'$ be a 
boundary coloring of $E^u(\Gamma)$. 
Then the number of $k$-admissible colorings compatible with $j'$ 
is given by the following formula. 
$$
d_{g,n}(k;j'(\Gamma)):=
^{\#}QCG_k(\Gamma;j')=
\left(\frac{k+2}{2}\right)^{g-1}\sum_{\nu=1}^{k+1}
\frac{\prod_{f_l\in E^u(\Gamma)}\left({\sin\frac{j_l+1}{k+2}\nu\pi}\right)}
{\left(\sin\frac{\nu\pi}{k+2}\right)^{2g-2+n}} \ . 
$$

\medskip

\begin{proof}[Proof of Proposition~\ref{main2}]
By using the factorization property and 
the Verlinde formula, 
we have the following computation:
%the number in the first parenthesis can be computed as follows. 
\begin{eqnarray*}
&&\sum_{j'(E^u_{\mu})}
(-1)^{\sum j'(E^u_{\mu})}
{^{\#}QCG_k(\Gamma'(\lambda)_{\mu};j'(\Gamma'(\lambda)_{\mu})))}\\
&=&\sum_{j'(E^u_{\mu})}
(-1)^{\sum j'(E^u_{\mu})}\left(\frac{k+2}{2}\right)^{g_2-m'-1}
\sum_{l=1}^{k+1}\frac{\prod_{E^t_{\lambda}\sqcup \overline{E^u_{\lambda}}}
\left(\sin\frac{j'_{\nu}+1}{k+2}l\pi\right)
\prod_{E^u_{\mu}}\left(\sin^2\frac{j^{'}_{\nu'}+1}{k+2}l\pi\right)}
{\left(\sin\frac{l\pi}{k+2}\right)^{2g_2-2+n}}\\  
%&&(E^u(\Gamma'(\lambda)_{\mu})=
%Ex^t(\lambda)\sqcup \overline{E^u_{\lambda}}(\Gamma)\sqcup E^u_{\mu})\\
&=&\left(\frac{k+2}{2}\right)^{g_2-m'-1}
\sum_{l=1}^{k+1}\frac{\prod_{E^t_{\lambda}\sqcup \overline{E^u_{\lambda}}}
\left(\sin\frac{j'_{\nu}+1}{k+2}l\pi\right)
\sum_{j'(E^u_{\mu})}
(-1)^{\sum j'(E^u_{\mu})}
\prod_{E^u_{\mu}}\left(\sin^2\frac{j^{'}_{\nu'}+1}{k+2}l\pi\right)}
{\left(\sin\frac{l\pi}{k+2}\right)^{2g_2+n-2}}\\
&=&\left(\frac{k+2}{2}\right)^{g_2-1}
\prod_{E^t_{\lambda}\sqcup \overline{E^u_{\lambda}}}
\sin\frac{j'_{\nu}+1}{2}\pi.
\end{eqnarray*}
Here we used the following elementary equalities 
in the last equal sign; 
\begin{eqnarray*}
&&\sum_{j'(E^u_{\mu})}(-1)^{\sum j'(E^u_{\mu})}
\prod_{E^u_{\mu}}\left(\sin^2\frac{j'_{\nu'}+1}{k+2}l\pi\right)\\
&=&\sum_{j'_1,\cdots,j'_m}
(-1)^{j'_1+\cdots+j'_m}\prod_{\nu'=1}^{m'}
\left(\sin^2\frac{j'_{\nu'}+1}{k+2}l\pi\right)
%&=&\left(\sum_{a=0}^{k}
%\left(\sin^2\frac{a+1}{k+2}l\pi-\sin^2\frac{a+2}{k+2}l\pi\right)\right)^{m'}\\
%&=&\left(\frac{1}{2}\sum_{a=0}^{k}
%\left(-\cos 2\frac{a+1}{k+2}l\pi+\cos 2\frac{a+2}{k+2}l\pi\right)\right)^{m'}
=\left\{\begin{array}{lll} \left(\frac{k+2}{2}\right)^{m'} \quad (l=\frac{k+2}{2})\\ 
 \\ 0 \quad (l\neq\frac{k+2}{2}).\end{array}\right.
\end{eqnarray*}

Then the number in the proposition 
can be computed as follows. 

\begin{eqnarray*}
&&\left(\frac{k+2}{2}\right)^{g_2-1}
\sum_{j'(E^t_{\lambda})}(-1)^{\sum j'(E^t_{\lambda})}
\prod_{E^t_{\lambda}\sqcup \overline{E^u_{\lambda}}}{\sin\frac{j'_{\nu}+1}{2}\pi}\\
&=&\left(\frac{k+2}{2}\right)^{g_2-1}
\prod_{\overline{E^u_{\lambda}}}{\sin\frac{j'_{\nu}+1}{2}\pi}
\left(\sum_{a=0}^k(-1)^{a}\sin\frac{a+1}{2}\pi\right)^m 
=\left(\frac{k+2}{2}\right)^{g_2-1+m}
\prod_{\overline{E^u_{\lambda}}}{\sin\frac{j'_l+1}{2}\pi}.
\end{eqnarray*}
\end{proof}

\begin{Rem}\label{characterization}
We can show that the external edge condition 
is a characterization of the Heisenberg action on $V_{2k+4}(C_{\Gamma})$ 
in the following sense. 
Let $\Gamma$ be a unitrivalent graph and $k$ a non-negative integer. 
Fix a boundary coloring $j'$ of $\Gamma$ and 
let $R_k(\Gamma;j')$ be the free $R_k$-module 
generated by the finite set $QCG_k(\Gamma;j')$, 
which is isomorphic to $V_{2k+4}(C_{\Gamma};j')$. 
Take a family of maps $\alpha=(\alpha_j:H_1(\Gamma;\Z/2)\to R_k)$ 
parameterized by $j\in QCG_k(\Gamma;j')$. 
We assume that $\alpha$ satisfies the {\it cocycle condition} 
$$
\alpha_j(\lambda_1+\lambda_2)=\alpha_{\lambda_2\cdot j}(\lambda_1)\alpha_j(\lambda_2)
\qquad  (j\in QCG_k(\Gamma;j'), \ \lambda_1,\lambda_2\in H_1(\Gamma;\Z/2)). 
$$
By using such $\alpha$ we can define a homomorphism 
\begin{eqnarray*}
\rho^{(k)}_{\alpha}:\E(\Gamma)&\to& GL(R_k(\Gamma;j'))\\ 
u^m\tau(\mu,\lambda)&\mapsto& 
\left(|j\rangle\mapsto 
A^{(k+2)^2(m+\mu\circ\lambda)}
(-1)^{(k+1)(m+\mu\circ\lambda)+j_\mu}\alpha_j(\lambda)|\lambda\cdot j\rangle\right).
\end{eqnarray*}
Our proof of Theorem~\ref{trace} shows that 
if $\alpha=(\alpha_j)$ satisfies the external edge condition, 
then the representation $\rho^{(k)}_{\alpha}$ on $R_k(\Gamma;j')$ is 
equivalent to the representation $\rho^{(k)}$ on $V_{2k+4}(C_{\Gamma};j')$. 
In a suitable categorical setting, 
we can show that the converse is true. 
%Then we have the following. 
%\begin{Prop}
%If the map $\alpha$ satisfies the cocycle condition and 
%is compatible with the factorization of the graph, 
%then the representation $\rho^{(k)}_{\alpha}$ on $R_k(\Gamma;j')$ 
%is equivalent to the representation $\rho^{(k)}$ on $V_{2k+4}(C_{\Gamma};j')$ 
%if and only if $\alpha=(\alpha_j)$ satisfies the external edge condition. 
%\end{Prop}
See \cite{extcond} for detail. 
\end{Rem}

%%%%%%%%%%%%%%%%%section%%%%%%%%%%%
\section{Brick decomposition}

Our trace formula (Theorem~\ref{trace}) 
enable one to describe the dimension 
of the direct summands of canonical decomposition of 
the TQFT-module $V_{2k+4}(C)$ associated with the $\E(C)$-action. 
Such a decomposition is called the {\it brick decomposition} in \cite{AndMas}. 

\medskip

\noindent
{\bf Theorem.} ([{\bf 5}, Theorem~1.6 ]) \ 
{\it
Let $(C,l,c)$ be a surface with colored structure. 

\begin{itemize}
\item[(i)]
If $k\equiv 0\bmod 4$, 
then the action of $\E(C)$ decomposes $V_{2k+4}(C)$ into a
direct sum of subspaces $V_{2k+4}(C,h)$, canonically associated with 
cohomology classes $h\in H^1(C-l;\Z/2)$. 

\item[(ii)]
If $k\equiv 2\bmod 4$, 
then the action of $\E(C)$ decomposes $V_{2k+4}(C)$ into a
direct sum of subspaces $V_{2k+4}(C,q)$, canonically associated with 
spin structures $q$ on  $C- l$. 
\end{itemize}
}

These subspaces are defined as follows; 
$$
V_{2k+4}(C,h)=\{v\in V_{2k+4}(C) \ | \ \rho^{(k)}(\tau(a))v=h(a)v \ 
{\rm for \ all} \ a\in H_1(C- l;\Z/2)\} 
$$
$$
V_{2k+4}(C,q)=\{v\in V_{2k+4}(C) \ | \ \rho^{(k)}(\tau(a))v=q(a)v \ 
{\rm for \ all} \ a\in H_1(C- l;\Z/2)\}. 
$$
Here we used identifications 
$H^1(C- l;\Z/2)=Hom(H_1(C- l;\Z/2),\Z/2)$ and 
$Spin(C- l)$ with the set of quadratic forms on 
$H_1(C- l;\Z/2)$ inducing the intersection form. 

\begin{Rem}
They also showed in \cite{AndMas} and \cite{BHMV} that 
if $k$ is odd, then 
$V_{2k+4}(C)$ is isomorphic 
to the tensor product of certain lower level modules 
as a representation of $\E(C)$. 
\end{Rem}

In \cite{AndMas} and \cite{BHMV} they gave the dimension formula for submodules 
$V_{2k+4}(C,h)$ and $V_{2k+4}(C,q)$ 
for a closed surface equipped with the empty link. 
Theorem~\ref{trace} tells us that the similar formulas 
hold for a surface with non-empty link. 
We demonstrate the computation as follows. 
The computation itself is the almost same as in \cite{AndMas}
except the consideration of the boundary. 

\medskip

Now we start from a ribbon unitrivalent graph $\Gamma$. 
Then we have the associated closed surface $C=C_{\Gamma}$ 
and the isomorphism $H_1(C^{\circ};\Z/2)=\Lmd\oplus\lmd$, 
where $C^{\circ}$ is the compact surface obtained by 
removing open discs around the banded link in $C$. 

\begin{Rem}
As it is noted in \cite{BHMV}, 
the action of $\mu=\sum\epsilon_le_l\in\Lambda_0^u/2\Lambda^u$ 
represented by boundary circles 
is given by scalar multiplication of $(-1)^{\sum\epsilon_lj'_l}$, 
and hence $V_{2k+4}(C,h)$ (resp. $V_{2k+4}(C,q)$) is 
zero except if $h(e_l)=(-1)^{j_l'}$ (resp. $q(e_l)=(-1)^{j_l'}$) 
for all boundary circles $e_l$. 
We denote by $H^1(C;j')$ (resp. $Spin(C;j')$) the set of 
cohomology classes (resp. spin structures) on $C-l$ 
satisfying $h(e_l)=(-1)^{j_l'}$ (resp. $q(e_l)=(-1)^{j_l'}$) 
for all $e_l\in \Lambda_0^u/2\Lambda^u$. 
\end{Rem}

\medskip

\noindent
{\bf The case $k\equiv 0$ mod $4$.}

\medskip

We have a direct sum decomposition 
$$
V_{2k+4}(C)=V_{2k+4}(C;j')=\bigoplus_{h}V_{2k+4}(C,h;j'). 
$$
The trace formula asserts that the trace takes the same value for all 
nontrivial non-boundary element in $H_1(C^{\circ};\Z/2)$, and is therefore 
invariant under the automorphism group of $H_1(C^{\circ};\Z/2)$. 
One can check that the action of automorphism group 
on $H^1(C;j')$ is transitive. 
These imply that the dimension of $V_{2k+4}(C:h)$ is same for 
all non-trivial $h\in H^1(C;j')$. 
We denote this dimension by $d_{g,n}^{(1)}(k;j')$, and 
we put $d_{g,n}^{(0)}(k;j'):=\dim V_{2k+4}(C,0;j')$. 
These numbers can be computed by the following two equalities; 
\begin{eqnarray*}
d_{g,n}(k;j')&=&d_{g,n}^{(0)}(k;j')+(2^{2g}-1)d_{g,n}^{(1)}(k;j')\\ 
\gamma(j')\left(\frac{k+2}{2}\right)^{g-1}&=&
Tr(\tau_a)=d_{g,n}^{(0)}(k;j')-d_{g,n}^{(1)}(k;j') \quad (a\notin \Lambda^u_0/2\Lambda^u). \\ 
\end{eqnarray*}
Note that $^{\#}H^1(C;j')=2^{2g}$ and 
$d_{g,n}(k;j')=\dim V_{2k+4}(C;j')$ is the Verlinde number. 
In this way for $k\equiv 0\bmod 4$, one has that  
\begin{eqnarray*}
d_{g,n}^{(0)}(k;j')&=&\frac{1}{2^{2g}}
\left(d_{g,n}(k;j')+(2^{2g}-1)\gamma(j')\left(\frac{k+2}{2}\right)^{g-1}\right) \\
d_{g,n}^{(1)}(k;j')&=&\frac{1}{2^{2g}}
\left(d_{g,n}(k;j')-\gamma(j')\left(\frac{k+2}{2}\right)^{g-1}\right). \\
\end{eqnarray*}

\medskip

\noindent
{\bf The case of $k\equiv 2$ mod $4$.}

\medskip

We have a direct sum decomposition 
$$
V_{2k+4}(C)=V_{2k+4}(C;j')=\bigoplus_{q}V_{2k+4}(C,q;j'). 
$$
As in the same way for $k\equiv 0\bmod 4$, 
the trace formula asserts that the trace is 
invariant under the automorphism group of $H_1(C^{\circ};\Z/2)$. 
Fix a decomposition $H_1(C^{\circ};\Z/2)=\Lambda_0^u/2\Lambda^u
\oplus \left(\Lambda_0^u/2\Lambda^u\right)^{\bot}$, 
$\left(\Lambda_0^u/2\Lambda^u\right)^{\bot}:=
H_1(C^{\circ};\Z/2)/\left(\Lambda_0^u/2\Lambda^u\right)$. 
According to this decomposition we have the decomposition 
$q(a)=q_1(a_1)q_2(a_2)$, for $a=(a_1,a_2)\in H_1(C^{\circ};\Z/2)$. 
Since it is enough to consider for 
elements in $Spin(C;j')$, we may assume that the map $q_1$ is equal to the map 
$\sum\varepsilon_le_l\mapsto (-1)^{\sum\varepsilon_lj'_l}$. 
On the other hand there is a well-defined {\it Arf invariant} Arf$(q_2)\in \Z/2$, 
which parameterizes isomorphism classes of quadratic forms. 
In this way one knows that the dimension of $V_{2k+4}(C,q;j')$ 
depends only on Arf$(q):=$Arf$(q_2)$. 
For $\varepsilon=+$ or $-$, 
we put $d_{g,n}^{\varepsilon}(k;j'):=\dim V_{2k+4}(C,q_{\varepsilon};j')$, 
where $q_{\varepsilon}$ has Arf invariant $\varepsilon 1\in \{\pm1\}$. 
These numbers can be computed by the following two equalities; 

\begin{eqnarray*}
d_{g,n}(k;j')&=&2^{g-1}(2^g+1)d_{g,n}^{+}(k;j')+2^{g-1}(2^g-1)d_{g,n}^{-}(k;j')\\ 
\gamma(j')\left(\frac{k+2}{2}\right)^{g-1}&=&
Tr(\tau(a))=2^{g-1}(d_{g,n}^{+}(k;j')-d_{g,n}^{-}(k;j')) 
\quad (a\notin \Lambda^u_0/2\Lambda^u). \\ 
\end{eqnarray*}

Here we used the facts that for $a\notin \Lambda^u_0/2\Lambda^u$ and 
$\varepsilon=+,-$, one has 
\begin{itemize}
\item ${^\#}\{q_2:\left(\Lambda_0^u/2\Lambda^u\right)^{\bot}\to\Z/2 
\ | \ {\rm Arf}(q_2)=0\}=2^{g-1}(2^{g}-1)$
\item ${^\#}\{q_2:\left(\Lambda_0^u/2\Lambda^u\right)^{\bot}\to\Z/2 
\ | \ q_2(a)=-1, {\rm Arf}(q_2)=\varepsilon\}=2^{2g-2}$.
\end{itemize}

In this way for $k\equiv 2\bmod 4$, one has that  
$$
d_{g,n}^{\varepsilon}(k;j')=
\frac{1}{2^{2g}}\left(d_{g,n}(k;j')+(\varepsilon 2^g-1)
\gamma(j')\left(\frac{k+2}{2}\right)^{g-1}\right). 
$$

\medskip 

{Acknowledgments.} 
The author is grateful to G.~Masbaum, 
who explained some details 
about skein theoretical computations to the author 
and pointed out that the description of coefficients $\delta_j$
could be simplified. 
The author would like to acknowledge CTQM, 
especially the director J.~E.~Andersen, 
for their hospitality, 
where the author attended the lecture by G.~Masbaum. 
The author is indebted to M.~Furuta 
for his encouragements.

\end{document}

%% file: pic25.tex
%WinTpicVersion3.08
\unitlength 0.1in
\begin{picture}( 18.1500,  6.6000)( 14.0400,-20.6000)
% LINE 1 0 3 0
% 2 1507 1400 1507 1707
% 
\special{pn 13}%
\special{pa 1508 1400}%
\special{pa 1508 1708}%
\special{fp}%
% LINE 1 0 3 1
% 2 1507 1707 1200 2014
% 
\special{pn 13}%
\special{pa 1508 1708}%
\special{pa 1200 2014}%
\special{fp}%
% LINE 1 0 3 2
% 2 1507 1707 1814 2014
% 
\special{pn 13}%
\special{pa 1508 1708}%
\special{pa 1814 2014}%
\special{fp}%
% LINE 2 2 3 0
% 2 2837 1400 2837 1707
% 
\special{pn 8}%
\special{pa 2838 1400}%
\special{pa 2838 1708}%
\special{dt 0.045}%
% LINE 2 2 3 1
% 2 2837 1707 2530 2014
% 
\special{pn 8}%
\special{pa 2838 1708}%
\special{pa 2530 2014}%
\special{dt 0.045}%
% LINE 2 2 3 2
% 2 2837 1707 3144 2014
% 
\special{pn 8}%
\special{pa 2838 1708}%
\special{pa 3144 2014}%
\special{dt 0.045}%
% ELLIPSE 1 0 3 0
% 4 2141 1405 2734 1990 2714 2252 3230 1405
% 
\special{pn 13}%
\special{ar 2142 1406 594 586  6.2831853 6.2831853}%
\special{ar 2142 1406 594 586  0.0000000 0.9825251}%
% ELLIPSE 1 0 3 0
% 4 2858 2270 3180 2582 3195 1948 2335 1810
% 
\special{pn 13}%
\special{ar 2858 2270 322 312  3.8789317 5.5052608}%
% ELLIPSE 1 0 3 0
% 4 3549 1410 2955 1995 2459 1410 2975 2257
% 
\special{pn 13}%
\special{ar 3550 1410 594 586  2.1593354 3.1415927}%
% CIRCLE 2 0 3 0
% 4 1507 1707 1554 1799 1630 1763 1599 1942
% 
\special{pn 8}%
\special{ar 1508 1708 104 104  1.1976482 6.2831853}%
\special{ar 1508 1708 104 104  0.0000000 0.4272399}%
% SARROW 2 0 3 1
% 2 1538 1806 1544 1803
% 
\special{pn 8}%
\special{pa 1538 1806}%
\special{pa 1544 1804}%
\special{fp}%
\special{sh 1}%
\special{pa 1544 1804}%
\special{pa 1476 1816}%
\special{pa 1496 1828}%
\special{pa 1494 1852}%
\special{pa 1544 1804}%
\special{fp}%
% CIRCLE 1 0 3 0
% 4 2920 2140 2742 2191 2916 1750 2794 1888
% 
\special{pn 13}%
\special{ar 2920 2140 186 186  4.2487414 4.7021329}%
% SARROW 1 0 3 1
% 2 2903 1956 2918 1955
% 
\special{pn 13}%
\special{pa 2904 1956}%
\special{pa 2918 1956}%
\special{fp}%
\special{sh 1}%
\special{pa 2918 1956}%
\special{pa 2850 1940}%
\special{pa 2866 1960}%
\special{pa 2854 1980}%
\special{pa 2918 1956}%
\special{fp}%
% CIRCLE 1 0 3 0
% 4 3185 1583 3289 1735 2816 1706 2985 1779
% 
\special{pn 13}%
\special{ar 3186 1584 184 184  2.3662952 2.8198421}%
% SARROW 1 0 3 1
% 2 3016 1655 3010 1641
% 
\special{pn 13}%
\special{pa 3016 1656}%
\special{pa 3010 1642}%
\special{fp}%
\special{sh 1}%
\special{pa 3010 1642}%
\special{pa 3018 1710}%
\special{pa 3032 1690}%
\special{pa 3056 1694}%
\special{pa 3010 1642}%
\special{fp}%
% CIRCLE 1 0 3 0
% 4 2483 1619 2562 1454 2775 1877 2753 1694
% 
\special{pn 13}%
\special{ar 2484 1620 184 184  0.2709469 0.7236585}%
% SARROW 1 0 3 1
% 2 2631 1729 2621 1741
% 
\special{pn 13}%
\special{pa 2632 1730}%
\special{pa 2622 1742}%
\special{fp}%
\special{sh 1}%
\special{pa 2622 1742}%
\special{pa 2680 1704}%
\special{pa 2656 1700}%
\special{pa 2648 1678}%
\special{pa 2622 1742}%
\special{fp}%
% POLYLINE 1 0 3 0
% 4 1920 1790 2350 1790 2350 1790 2350 1790
% 
\special{pn 13}%
\special{pa 1920 1790}%
\special{pa 2350 1790}%
\special{pa 2350 1790}%
\special{pa 2350 1790}%
\special{fp}%
% SARROW 1 0 3 1
% 2 1949 1790 1920 1790
% 
\special{pn 13}%
\special{pa 1950 1790}%
\special{pa 1920 1790}%
\special{fp}%
\special{sh 1}%
\special{pa 1920 1790}%
\special{pa 1988 1810}%
\special{pa 1974 1790}%
\special{pa 1988 1770}%
\special{pa 1920 1790}%
\special{fp}%
% SARROW 1 0 3 2
% 2 2321 1790 2350 1790
% 
\special{pn 13}%
\special{pa 2322 1790}%
\special{pa 2350 1790}%
\special{fp}%
\special{sh 1}%
\special{pa 2350 1790}%
\special{pa 2284 1770}%
\special{pa 2298 1790}%
\special{pa 2284 1810}%
\special{pa 2350 1790}%
\special{fp}%
\end{picture}%

%% file: pic2.tex
%WinTpicVersion3.08
\unitlength 0.1in
\begin{picture}( 22.1200,  8.9000)( 11.6000,-14.6000)
% LINE 2 0 3 0
% 2 1200 768 1200 1376
% 
\special{pn 8}%
\special{pa 1200 768}%
\special{pa 1200 1376}%
\special{fp}%
% LINE 2 0 3 0
% 2 1512 768 1512 1376
% 
\special{pn 8}%
\special{pa 1512 768}%
\special{pa 1512 1376}%
\special{fp}%
% STR 2 0 3 0
% 3 1160 1496 1160 1576 2 0
% $a$
\put(11.6000,-15.7600){\makebox(0,0)[lb]{$a$}}%
% STR 2 0 3 0
% 3 1480 1496 1480 1576 2 0
% $b$
\put(14.8000,-15.7600){\makebox(0,0)[lb]{$b$}}%
% STR 2 0 3 0
% 3 1784 1072 1784 1152 2 0
% $=$
\put(17.8400,-11.5200){\makebox(0,0)[lb]{$=$}}%
% LINE 2 0 3 0
% 2 2780 764 3076 932
% 
\special{pn 8}%
\special{pa 2780 764}%
\special{pa 3076 932}%
\special{fp}%
% LINE 2 0 3 0
% 2 3372 764 3076 932
% 
\special{pn 8}%
\special{pa 3372 764}%
\special{pa 3076 932}%
\special{fp}%
% LINE 2 0 3 0
% 2 2780 1460 3076 1292
% 
\special{pn 8}%
\special{pa 2780 1460}%
\special{pa 3076 1292}%
\special{fp}%
% LINE 2 0 3 0
% 2 3372 1460 3076 1292
% 
\special{pn 8}%
\special{pa 3372 1460}%
\special{pa 3076 1292}%
\special{fp}%
% LINE 2 0 3 0
% 2 3076 924 3076 1292
% 
\special{pn 8}%
\special{pa 3076 924}%
\special{pa 3076 1292}%
\special{fp}%
% STR 2 0 3 0
% 3 2740 1500 2740 1580 2 0
% $a$
\put(27.4000,-15.8000){\makebox(0,0)[lb]{$a$}}%
% STR 2 0 3 0
% 3 3308 1500 3308 1580 2 0
% $b$
\put(33.0800,-15.8000){\makebox(0,0)[lb]{$b$}}%
% STR 2 0 3 0
% 3 2048 1232 2048 1312 2 0
% $\displaystyle\sum_{c}\frac{\langle c\rangle}{\langle a,b,c\rangle}$
\put(20.4800,-13.1200){\makebox(0,0)[lb]{$\displaystyle\sum_{c}\frac{\langle c\rangle}{\langle a,b,c\rangle}$}}%
% STR 2 0 3 0
% 3 2780 660 2780 740 2 0
% $a$
\put(27.8000,-7.4000){\makebox(0,0)[lb]{$a$}}%
% STR 2 0 3 0
% 3 3348 660 3348 740 2 0
% $b$
\put(33.4800,-7.4000){\makebox(0,0)[lb]{$b$}}%
% STR 2 0 3 0
% 3 3148 1092 3148 1172 2 0
% $c$
\put(31.4800,-11.7200){\makebox(0,0)[lb]{$c$}}%
\end{picture}%

%% file: pic11.tex
%WinTpicVersion3.08
\unitlength 0.1in
\begin{picture}( 19.2600, 10.9000)(  9.8000,-17.0000)
% ELLIPSE 2 0 3 0
% 4 1320 1190 1469 1453 1063 1196 1077 1190
% 
\special{pn 8}%
\special{ar 1320 1190 150 264  3.1415927 6.2831853}%
\special{ar 1320 1190 150 264  0.0000000 3.1299200}%
% LINE 2 0 3 0
% 2 1320 1680 1320 1456
% 
\special{pn 8}%
\special{pa 1320 1680}%
\special{pa 1320 1456}%
\special{fp}%
% LINE 2 0 3 0
% 2 1315 921 1315 697
% 
\special{pn 8}%
\special{pa 1316 922}%
\special{pa 1316 698}%
\special{fp}%
% STR 2 0 3 0
% 3 1396 716 1396 780 2 0
% $c$
\put(13.9600,-7.8000){\makebox(0,0)[lb]{$c$}}%
% STR 2 0 3 0
% 3 1390 1606 1390 1670 2 0
% $c'$
\put(13.9000,-16.7000){\makebox(0,0)[lb]{$c'$}}%
% STR 2 0 3 0
% 3 1502 1234 1502 1298 2 0
% $b$
\put(15.0200,-12.9800){\makebox(0,0)[lb]{$b$}}%
% STR 2 0 3 0
% 3 980 1226 980 1290 2 0
% $a$
\put(9.8000,-12.9000){\makebox(0,0)[lb]{$a$}}%
% STR 2 0 3 0
% 3 1822 1253 1822 1317 2 0
% $=$
\put(18.2200,-13.1700){\makebox(0,0)[lb]{$=$}}%
% STR 2 0 3 0
% 3 1990 1306 1990 1370 2 0
% $\delta_{c,c'}\frac{\langle a, \ b, \ c\rangle}{\langle c\rangle}$
\put(19.9000,-13.7000){\makebox(0,0)[lb]{$\delta_{c,c'}\frac{\langle a, \ b, \ c\rangle}{\langle c\rangle}$}}%
% LINE 2 0 3 0
% 2 2810 675 2810 1700
% 
\special{pn 8}%
\special{pa 2810 676}%
\special{pa 2810 1700}%
\special{fp}%
% STR 2 0 3 0
% 3 2906 726 2906 790 2 0
% $c$
\put(29.0600,-7.9000){\makebox(0,0)[lb]{$c$}}%
\end{picture}%

%% file: pic3.tex
%WinTpicVersion3.08
\unitlength 0.1in
\begin{picture}( 29.8000,  9.2600)(  8.2000,-16.6600)
% LINE 2 0 3 0
% 2 1076 1410 820 1666
% 
\special{pn 8}%
\special{pa 1076 1410}%
\special{pa 820 1666}%
\special{fp}%
% LINE 2 0 3 1
% 2 1588 1410 1844 1666
% 
\special{pn 8}%
\special{pa 1588 1410}%
\special{pa 1844 1666}%
\special{fp}%
% LINE 2 0 3 2
% 2 1332 1026 1332 770
% 
\special{pn 8}%
\special{pa 1332 1026}%
\special{pa 1332 770}%
\special{fp}%
% LINE 2 0 3 3
% 4 1332 1026 1076 1410 1588 1410 1332 1026
% 
\special{pn 8}%
\special{pa 1332 1026}%
\special{pa 1076 1410}%
\special{fp}%
\special{pa 1588 1410}%
\special{pa 1332 1026}%
\special{fp}%
% LINE 2 0 3 4
% 2 1070 1410 1582 1410
% 
\special{pn 8}%
\special{pa 1070 1410}%
\special{pa 1582 1410}%
\special{fp}%
% STR 2 0 3 0
% 3 1370 830 1370 910 2 0
% $a$
\put(13.7000,-9.1000){\makebox(0,0)[lb]{$a$}}%
% STR 2 0 3 0
% 3 1070 1130 1070 1210 2 0
% $f$
\put(10.7000,-12.1000){\makebox(0,0)[lb]{$f$}}%
% STR 2 0 3 0
% 3 1500 1110 1500 1190 2 0
% $e$
\put(15.0000,-11.9000){\makebox(0,0)[lb]{$e$}}%
% STR 2 0 3 0
% 3 1792 1480 1792 1560 2 0
% $c$
\put(17.9200,-15.6000){\makebox(0,0)[lb]{$c$}}%
% STR 2 0 3 0
% 3 1280 1510 1280 1590 2 0
% $d$
\put(12.8000,-15.9000){\makebox(0,0)[lb]{$d$}}%
% STR 2 0 3 0
% 3 830 1450 830 1530 2 0
% $b$
\put(8.3000,-15.3000){\makebox(0,0)[lb]{$b$}}%
% STR 2 0 3 0
% 3 2040 1160 2040 1240 2 0
% $=$
\put(20.4000,-12.4000){\makebox(0,0)[lb]{$=$}}%
% LINE 2 0 3 0
% 6 3480 1240 3480 920 3160 1560 3480 1240 3480 1240 3800 1560
% 
\special{pn 8}%
\special{pa 3480 1240}%
\special{pa 3480 920}%
\special{fp}%
\special{pa 3160 1560}%
\special{pa 3480 1240}%
\special{fp}%
\special{pa 3480 1240}%
\special{pa 3800 1560}%
\special{fp}%
% STR 2 0 3 0
% 3 3570 930 3570 1010 2 0
% $a$
\put(35.7000,-10.1000){\makebox(0,0)[lb]{$a$}}%
% STR 2 0 3 0
% 3 3800 1320 3800 1400 2 0
% $c$
\put(38.0000,-14.0000){\makebox(0,0)[lb]{$c$}}%
% STR 2 0 3 0
% 3 3160 1320 3160 1400 2 0
% $b$
\put(31.6000,-14.0000){\makebox(0,0)[lb]{$b$}}%
% STR 2 0 3 0
% 3 2300 1280 2300 1360 2 0
% $\frac{\left<\begin{array}{ccc}a & b &c\\ d & e & f\end{array}\right>}{\displaystyle{\langle a \ b \ c\rangle}}$
\put(23.0000,-13.6000){\makebox(0,0)[lb]{$\frac{\left<\begin{array}{ccc}a & b &c\\ d & e & f\end{array}\right>}{\displaystyle{\langle a \ b \ c\rangle}}$}}%
\end{picture}%

%% file: pic4.tex
%WinTpicVersion3.08
\unitlength 0.1in
\begin{picture}( 18.4000,  8.9000)(  8.1000,-15.9000)
% LINE 2 0 3 0
% 4 1104 1270 1104 1590 1104 1590 1104 1590
% 
\special{pn 8}%
\special{pa 1104 1270}%
\special{pa 1104 1590}%
\special{fp}%
\special{pa 1104 1590}%
\special{pa 1104 1590}%
\special{fp}%
% CIRCLE 2 0 3 0
% 4 1104 1270 1104 1270 944 1118 1424 790
% 
\special{pn 8}%
\special{ar 1104 1270 0 0  5.3003916 6.2831853}%
\special{ar 1104 1270 0 0  0.0000000 3.9013554}%
% STR 2 0 3 0
% 3 1424 1190 1424 1270 2 0
% $= \ \delta(c;a,b)$ 
\put(14.2400,-12.7000){\makebox(0,0)[lb]{$= \ \delta(c;a,b)$ }}%
% LINE 2 0 3 0
% 2 2384 1590 2384 1270
% 
\special{pn 8}%
\special{pa 2384 1590}%
\special{pa 2384 1270}%
\special{fp}%
% STR 2 0 3 0
% 3 1176 1510 1176 1590 2 0
% $c$
\put(11.7600,-15.9000){\makebox(0,0)[lb]{$c$}}%
% STR 2 0 3 0
% 3 810 790 810 870 2 0
% $a$
\put(8.1000,-8.7000){\makebox(0,0)[lb]{$a$}}%
% STR 2 0 3 0
% 3 1348 790 1348 870 2 0
% $b$
\put(13.4800,-8.7000){\makebox(0,0)[lb]{$b$}}%
% STR 2 0 3 0
% 3 2630 820 2630 900 2 0
% $b$
\put(26.3000,-9.0000){\makebox(0,0)[lb]{$b$}}%
% STR 2 0 3 0
% 3 2100 820 2100 900 2 0
% $a$
\put(21.0000,-9.0000){\makebox(0,0)[lb]{$a$}}%
% STR 2 0 3 0
% 3 2456 1510 2456 1590 2 0
% $c$
\put(24.5600,-15.9000){\makebox(0,0)[lb]{$c$}}%
% ELLIPSE 2 0 3 0
% 4 900 1120 1183 1345 1089 1249 712 670
% 
\special{pn 8}%
\special{ar 900 1120 284 226  4.3916989 6.2831853}%
\special{ar 900 1120 284 226  0.0000000 0.7086263}%
% BOX 2 5 2 0
% 2 1200 910 1040 1070
% 
\special{pn 8}%
\special{sh 0}%
\special{pa 1200 910}%
\special{pa 1040 910}%
\special{pa 1040 1070}%
\special{pa 1200 1070}%
\special{pa 1200 910}%
\special{ip}%
% ELLIPSE 2 0 3 0
% 4 1320 1130 1037 1355 1508 680 1131 1259
% 
\special{pn 8}%
\special{ar 1320 1130 284 226  2.4329664 5.0330791}%
% LINE 2 0 3 0
% 2 2380 1290 2110 950
% 
\special{pn 8}%
\special{pa 2380 1290}%
\special{pa 2110 950}%
\special{fp}%
% LINE 2 0 3 0
% 2 2380 1290 2650 950
% 
\special{pn 8}%
\special{pa 2380 1290}%
\special{pa 2650 950}%
\special{fp}%
\end{picture}%

%% file: pic26.tex
%WinTpicVersion3.08
\unitlength 0.1in
\begin{picture}( 36.7000,  9.5800)( 16.6000,-19.6800)
% ELLIPSE 2 2 3 0
% 4 4754 1726 4698 1935 4754 2287 4750 1080
% 
\special{pn 8}%
\special{ar 4754 1726 56 210  4.6892717 4.7798377}%
\special{ar 4754 1726 56 210  5.0515359 5.1421019}%
\special{ar 4754 1726 56 210  5.4138000 5.5043661}%
\special{ar 4754 1726 56 210  5.7760642 5.8666302}%
\special{ar 4754 1726 56 210  6.1383283 6.2288944}%
\special{ar 4754 1726 56 210  6.5005925 6.5911585}%
\special{ar 4754 1726 56 210  6.8628566 6.9534227}%
\special{ar 4754 1726 56 210  7.2251208 7.3156868}%
\special{ar 4754 1726 56 210  7.5873849 7.6779510}%
% ELLIPSE 2 0 3 0
% 4 4759 1735 4810 1942 4764 1087 4759 2295
% 
\special{pn 8}%
\special{ar 4760 1736 52 208  1.5707963 4.7436288}%
% ELLIPSE 2 2 3 0
% 4 4065 1704 4007 1891 4065 2206 4062 1124
% 
\special{pn 8}%
\special{ar 4066 1704 58 188  4.6957239 4.7936830}%
\special{ar 4066 1704 58 188  5.0875606 5.1855198}%
\special{ar 4066 1704 58 188  5.4793973 5.5773565}%
\special{ar 4066 1704 58 188  5.8712341 5.9691932}%
\special{ar 4066 1704 58 188  6.2630708 6.3610300}%
\special{ar 4066 1704 58 188  6.6549075 6.7528667}%
\special{ar 4066 1704 58 188  7.0467443 7.1447034}%
\special{ar 4066 1704 58 188  7.4385810 7.5365402}%
\special{ar 4066 1704 58 188  7.8304177 7.8539816}%
% ELLIPSE 2 0 3 0
% 4 4071 1712 4125 1898 4074 1129 4071 2215
% 
\special{pn 8}%
\special{ar 4072 1712 54 186  1.5707963 4.7301386}%
% SPLINE 2 2 3 0
% 41 5316 1493 5316 1499 5313 1502 5310 1507 5306 1509 5300 1514 5294 1517 5288 1522 5278 1526 5269 1528 5258 1532 5246 1534 5234 1537 5220 1540 5206 1541 5190 1544 5175 1545 5159 1546 5143 1548 5127 1548 5111 1548 5094 1548 5079 1548 5062 1546 5046 1545 5031 1544 5015 1543 5002 1541 4986 1538 4976 1536 4964 1532 4951 1529 4941 1527 4932 1523 4923 1519 4917 1516 4912 1512 4908 1508 4906 1504 4902 1500 4900 1494
% 
\special{pn 8}%
\special{pa 5316 1494}%
\special{pa 5298 1516}%
\special{pa 5268 1528}%
\special{pa 5238 1536}%
\special{pa 5206 1542}%
\special{pa 5174 1546}%
\special{pa 5142 1548}%
\special{pa 5110 1548}%
\special{pa 5078 1548}%
\special{pa 5046 1546}%
\special{pa 5014 1544}%
\special{pa 4984 1538}%
\special{pa 4952 1530}%
\special{pa 4922 1518}%
\special{pa 4900 1496}%
\special{pa 4900 1494}%
\special{sp -0.045}%
% SPLINE 2 0 3 0
% 41 4897 1484 4897 1480 4899 1476 4902 1473 4906 1470 4912 1465 4918 1462 4926 1460 4935 1456 4946 1455 4957 1450 4970 1447 4983 1445 4996 1443 5010 1442 5025 1440 5042 1440 5057 1438 5073 1438 5089 1436 5105 1436 5121 1438 5138 1438 5153 1438 5169 1440 5185 1441 5200 1442 5214 1444 5226 1446 5240 1449 5253 1451 5262 1455 5273 1456 5282 1460 5288 1464 5294 1467 5301 1471 5303 1475 5310 1478 5310 1483 5310 1486
% 
\special{pn 8}%
\special{pa 4898 1484}%
\special{pa 4918 1462}%
\special{pa 4948 1454}%
\special{pa 4980 1446}%
\special{pa 5012 1442}%
\special{pa 5044 1440}%
\special{pa 5076 1438}%
\special{pa 5106 1436}%
\special{pa 5138 1438}%
\special{pa 5170 1440}%
\special{pa 5202 1442}%
\special{pa 5234 1448}%
\special{pa 5264 1456}%
\special{pa 5294 1468}%
\special{pa 5310 1486}%
\special{sp}%
% ELLIPSE 2 2 3 0
% 4 4419 1483 4359 1956 4419 2342 4419 921
% 
\special{pn 8}%
\special{ar 4420 1484 60 474  4.7123890 4.7574171}%
\special{ar 4420 1484 60 474  4.8925016 4.9375297}%
\special{ar 4420 1484 60 474  5.0726141 5.1176423}%
\special{ar 4420 1484 60 474  5.2527267 5.2977548}%
\special{ar 4420 1484 60 474  5.4328393 5.4778674}%
\special{ar 4420 1484 60 474  5.6129518 5.6579800}%
\special{ar 4420 1484 60 474  5.7930644 5.8380925}%
\special{ar 4420 1484 60 474  5.9731770 6.0182051}%
\special{ar 4420 1484 60 474  6.1532895 6.1983177}%
\special{ar 4420 1484 60 474  6.3334021 6.3784303}%
\special{ar 4420 1484 60 474  6.5135147 6.5585428}%
\special{ar 4420 1484 60 474  6.6936273 6.7386554}%
\special{ar 4420 1484 60 474  6.8737398 6.9187680}%
\special{ar 4420 1484 60 474  7.0538524 7.0988805}%
\special{ar 4420 1484 60 474  7.2339650 7.2789931}%
\special{ar 4420 1484 60 474  7.4140775 7.4591057}%
\special{ar 4420 1484 60 474  7.5941901 7.6392182}%
\special{ar 4420 1484 60 474  7.7743027 7.8193308}%
% ELLIPSE 2 0 3 0
% 4 4419 1489 4481 1963 4419 928 4419 2348
% 
\special{pn 8}%
\special{ar 4420 1490 62 474  1.5707963 4.7123890}%
% ELLIPSE 2 2 3 0
% 4 4072 1246 4006 1419 4072 1713 4065 707
% 
\special{pn 8}%
\special{ar 4072 1246 66 174  4.6784215 4.7788399}%
\special{ar 4072 1246 66 174  5.0800951 5.1805135}%
\special{ar 4072 1246 66 174  5.4817687 5.5821872}%
\special{ar 4072 1246 66 174  5.8834424 5.9838608}%
\special{ar 4072 1246 66 174  6.2851160 6.3855344}%
\special{ar 4072 1246 66 174  6.6867897 6.7872081}%
\special{ar 4072 1246 66 174  7.0884633 7.1888817}%
\special{ar 4072 1246 66 174  7.4901369 7.5905554}%
% ELLIPSE 2 0 3 0
% 4 4078 1254 4139 1427 4083 715 4078 1720
% 
\special{pn 8}%
\special{ar 4078 1254 62 174  1.5707963 4.7386987}%
% ELLIPSE 1 0 3 0
% 4 4083 1504 4202 1414 4196 1476 3898 1459
% 
\special{pn 13}%
\special{ar 4084 1504 120 90  3.4552133 5.9667539}%
% ELLIPSE 1 0 3 0
% 4 4083 1395 4250 1518 3825 1456 4240 1432
% 
\special{pn 13}%
\special{ar 4084 1396 168 124  0.3083157 2.8303435}%
% ELLIPSE 1 0 3 0
% 4 4769 1513 4888 1425 4880 1488 4582 1469
% 
\special{pn 13}%
\special{ar 4770 1514 120 88  3.4472156 5.9859530}%
% ELLIPSE 1 0 3 0
% 4 4769 1404 4935 1528 4510 1467 4926 1442
% 
\special{pn 13}%
\special{ar 4770 1404 166 124  0.3140879 2.8279720}%
% ELLIPSE 2 2 3 0
% 4 4763 1232 4714 1421 4763 1742 4759 644
% 
\special{pn 8}%
\special{ar 4764 1232 50 190  4.6860793 4.7869196}%
\special{ar 4764 1232 50 190  5.0894406 5.1902809}%
\special{ar 4764 1232 50 190  5.4928020 5.5936423}%
\special{ar 4764 1232 50 190  5.8961633 5.9970036}%
\special{ar 4764 1232 50 190  6.2995246 6.4003650}%
\special{ar 4764 1232 50 190  6.7028860 6.8037263}%
\special{ar 4764 1232 50 190  7.1062473 7.2070877}%
\special{ar 4764 1232 50 190  7.5096087 7.6104490}%
% ELLIPSE 2 0 3 0
% 4 4766 1240 4814 1428 4772 650 4766 1749
% 
\special{pn 8}%
\special{ar 4766 1240 48 188  1.5707963 4.7521032}%
% ELLIPSE 2 0 3 0
% 4 4072 1460 4332 1714 4445 1460 4462 1460
% 
\special{pn 8}%
\special{ar 4072 1460 260 254  0.0000000 6.2831853}%
% LINE 2 0 3 0
% 2 4340 1482 4540 1482
% 
\special{pn 8}%
\special{pa 4340 1482}%
\special{pa 4540 1482}%
\special{fp}%
% ELLIPSE 2 0 3 0
% 4 4774 1469 5006 1733 5109 1469 5122 1469
% 
\special{pn 8}%
\special{ar 4774 1470 232 264  0.0000000 6.2831853}%
% LINE 1 0 3 0
% 2 2800 1603 2976 1712
% 
\special{pn 13}%
\special{pa 2800 1604}%
\special{pa 2976 1712}%
\special{fp}%
% LINE 1 0 3 0
% 2 2812 1244 2976 1060
% 
\special{pn 13}%
\special{pa 2812 1244}%
\special{pa 2976 1060}%
\special{fp}%
% ELLIPSE 1 0 3 0
% 4 2156 1394 2346 1648 2430 1394 2442 1394
% 
\special{pn 13}%
\special{ar 2156 1394 190 254  0.0000000 6.2831853}%
% LINE 1 0 3 0
% 2 1660 1420 1970 1420
% 
\special{pn 13}%
\special{pa 1660 1420}%
\special{pa 1970 1420}%
\special{fp}%
% LINE 1 0 3 0
% 2 2349 1415 2478 1415
% 
\special{pn 13}%
\special{pa 2350 1416}%
\special{pa 2478 1416}%
\special{fp}%
% ELLIPSE 1 0 3 0
% 4 2670 1415 2862 1670 2946 1415 2958 1415
% 
\special{pn 13}%
\special{ar 2670 1416 192 256  0.0000000 6.2831853}%
% POLYLINE 1 0 3 0
% 4 3076 1436 3436 1436 3436 1436 3436 1436
% 
\special{pn 13}%
\special{pa 3076 1436}%
\special{pa 3436 1436}%
\special{pa 3436 1436}%
\special{pa 3436 1436}%
\special{fp}%
% SARROW 1 0 3 1
% 2 3100 1436 3076 1436
% 
\special{pn 13}%
\special{pa 3100 1436}%
\special{pa 3076 1436}%
\special{fp}%
\special{sh 1}%
\special{pa 3076 1436}%
\special{pa 3144 1456}%
\special{pa 3130 1436}%
\special{pa 3144 1416}%
\special{pa 3076 1436}%
\special{fp}%
% SARROW 1 0 3 2
% 2 3412 1436 3436 1436
% 
\special{pn 13}%
\special{pa 3412 1436}%
\special{pa 3436 1436}%
\special{fp}%
\special{sh 1}%
\special{pa 3436 1436}%
\special{pa 3370 1416}%
\special{pa 3384 1436}%
\special{pa 3370 1456}%
\special{pa 3436 1436}%
\special{fp}%
% ELLIPSE 1 0 3 0
% 4 4495 1490 5330 1968 5451 1490 5388 1495
% 
\special{pn 13}%
\special{ar 4496 1490 836 478  0.0100780 6.2831853}%
% ELLIPSE 2 0 3 0
% 4 5148 1392 5358 1460 6438 808 3100 885
% 
\special{pn 8}%
\special{ar 5148 1392 210 68  3.7944030 5.3331626}%
% SPLINE 2 0 3 0
% 41 3668 1442 3668 1439 3669 1436 3670 1432 3671 1429 3674 1427 3675 1424 3677 1421 3679 1418 3683 1416 3686 1412 3688 1411 3693 1407 3696 1405 3700 1403 3704 1401 3709 1399 3712 1397 3717 1396 3722 1393 3727 1392 3732 1391 3738 1390 3743 1388 3748 1388 3754 1388 3758 1388 3764 1388 3770 1388 3774 1388 3780 1388 3785 1388 3790 1390 3794 1391 3800 1392 3804 1392 3810 1393 3814 1396 3818 1398 3822 1400 3826 1402
% 
\special{pn 8}%
\special{pa 3668 1442}%
\special{pa 3684 1416}%
\special{pa 3710 1400}%
\special{pa 3740 1390}%
\special{pa 3772 1388}%
\special{pa 3804 1392}%
\special{pa 3826 1402}%
\special{sp}%
% SPLINE 2 0 3 0
% 41 4946 1632 4952 1632 4958 1634 4965 1634 4971 1635 4977 1636 4984 1636 4990 1637 4998 1639 5005 1641 5012 1643 5019 1644 5026 1646 5034 1648 5042 1649 5049 1652 5058 1653 5065 1656 5073 1657 5081 1661 5089 1663 5097 1665 5105 1668 5113 1671 5120 1673 5129 1676 5137 1678 5144 1682 5151 1686 5158 1688 5166 1692 5174 1695 5180 1697 5188 1701 5195 1704 5202 1707 5209 1711 5214 1714 5221 1718 5228 1720 5234 1724
% 
\special{pn 8}%
\special{pa 4946 1632}%
\special{pa 4978 1636}%
\special{pa 5010 1642}%
\special{pa 5040 1650}%
\special{pa 5072 1658}%
\special{pa 5102 1666}%
\special{pa 5132 1678}%
\special{pa 5160 1690}%
\special{pa 5190 1702}%
\special{pa 5218 1718}%
\special{pa 5234 1724}%
\special{sp}%
\end{picture}%

%% file: pic21.tex
%WinTpicVersion3.08
\unitlength 0.1in
\begin{picture}( 12.8300,  5.7800)(  7.7000, -9.9700)
% STR 2 0 3 0
% 3 1674 531 1674 589 2 0
% $\mu$
\put(16.7400,-5.8900){\makebox(0,0)[lb]{$\mu$}}%
% LINE 2 0 3 0
% 2 770 775 2053 775
% 
\special{pn 8}%
\special{pa 770 776}%
\special{pa 2054 776}%
\special{fp}%
% BOX 2 5 2 0
% 2 1645 717 1784 845
% 
\special{pn 8}%
\special{sh 0}%
\special{pa 1646 718}%
\special{pa 1784 718}%
\special{pa 1784 846}%
\special{pa 1646 846}%
\special{pa 1646 718}%
\special{ip}%
% ELLIPSE 0 0 3 0
% 4 1476 781 1231 997 1108 845 1061 659
% 
\special{pn 20}%
\special{ar 1476 782 246 216  3.4626201 6.2831853}%
\special{ar 1476 782 246 216  0.0000000 2.9457653}%
% STR 2 0 3 0
% 3 776 694 776 752 2 0
% $a$
\put(7.7600,-7.5200){\makebox(0,0)[lb]{$a$}}%
\end{picture}%

%% file: pic22.tex
%WinTpicVersion3.08
\unitlength 0.1in
\begin{picture}( 37.3300, 11.9600)(  5.8000,-19.1600)
% LINE 2 0 3 0
% 2 580 957 1847 957
% 
\special{pn 8}%
\special{pa 580 958}%
\special{pa 1848 958}%
\special{fp}%
% BOX 2 5 2 0
% 2 1444 899 1583 1026
% 
\special{pn 8}%
\special{sh 0}%
\special{pa 1444 900}%
\special{pa 1584 900}%
\special{pa 1584 1026}%
\special{pa 1444 1026}%
\special{pa 1444 900}%
\special{ip}%
% ELLIPSE 0 0 3 0
% 4 1277 963 1035 1177 914 1026 868 842
% 
\special{pn 20}%
\special{ar 1278 964 242 214  3.4648095 6.2831853}%
\special{ar 1278 964 242 214  0.0000000 2.9484388}%
% STR 2 0 3 0
% 3 610 868 610 940 2 0
% $a$
\put(6.1000,-9.4000){\makebox(0,0)[lb]{$a$}}%
% BOX 2 5 2 0
% 2 3907 934 4046 1061
% 
\special{pn 8}%
\special{sh 0}%
\special{pa 3908 934}%
\special{pa 4046 934}%
\special{pa 4046 1062}%
\special{pa 3908 1062}%
\special{pa 3908 934}%
\special{ip}%
% ELLIPSE 0 0 3 0
% 4 2992 975 3171 1096 3476 975 2716 907
% 
\special{pn 20}%
\special{ar 2992 976 180 122  3.4923985 6.2831853}%
% STR 2 0 3 0
% 3 2623 881 2623 939 2 0
% $a$
\put(26.2300,-9.3900){\makebox(0,0)[lb]{$a$}}%
% STR 2 0 3 0
% 3 4110 883 4110 940 2 0
% $a$
\put(41.1000,-9.4000){\makebox(0,0)[lb]{$a$}}%
% STR 2 0 3 0
% 3 3208 881 3208 939 2 0
% $k-a$
\put(32.0800,-9.3900){\makebox(0,0)[lb]{$k-a$}}%
% ELLIPSE 0 0 3 0
% 4 3366 1565 3781 1726 3005 1448 3698 1458
% 
\special{pn 20}%
\special{ar 3366 1566 416 162  5.5896333 6.2831853}%
\special{ar 3366 1566 416 162  0.0000000 3.8382352}%
% STR 2 0 3 0
% 3 3223 1361 3223 1419 2 0
% $k-a$
\put(32.2300,-14.1900){\makebox(0,0)[lb]{$k-a$}}%
% STR 2 0 3 0
% 3 2730 1352 2730 1410 2 0
% $a$
\put(27.3000,-14.1000){\makebox(0,0)[lb]{$a$}}%
% STR 2 0 3 0
% 3 3915 1356 3915 1413 2 0
% $a$
\put(39.1500,-14.1300){\makebox(0,0)[lb]{$a$}}%
% LINE 2 0 3 0
% 2 2690 1905 4137 1905
% 
\special{pn 8}%
\special{pa 2690 1906}%
\special{pa 4138 1906}%
\special{fp}%
% STR 2 0 3 0
% 3 2730 1822 2730 1880 2 0
% $a$
\put(27.3000,-18.8000){\makebox(0,0)[lb]{$a$}}%
% VECTOR 1 0 3 0
% 2 2002 957 2362 957
% 
\special{pn 13}%
\special{pa 2002 958}%
\special{pa 2362 958}%
\special{fp}%
\special{sh 1}%
\special{pa 2362 958}%
\special{pa 2296 938}%
\special{pa 2310 958}%
\special{pa 2296 978}%
\special{pa 2362 958}%
\special{fp}%
% VECTOR 1 0 3 0
% 2 2011 1452 2371 1452
% 
\special{pn 13}%
\special{pa 2012 1452}%
\special{pa 2372 1452}%
\special{fp}%
\special{sh 1}%
\special{pa 2372 1452}%
\special{pa 2304 1432}%
\special{pa 2318 1452}%
\special{pa 2304 1472}%
\special{pa 2372 1452}%
\special{fp}%
% VECTOR 1 0 3 0
% 2 2011 1911 2371 1911
% 
\special{pn 13}%
\special{pa 2012 1912}%
\special{pa 2372 1912}%
\special{fp}%
\special{sh 1}%
\special{pa 2372 1912}%
\special{pa 2304 1892}%
\special{pa 2318 1912}%
\special{pa 2304 1932}%
\special{pa 2372 1912}%
\special{fp}%
% LINE 2 0 3 0
% 2 2506 984 4313 984
% 
\special{pn 8}%
\special{pa 2506 984}%
\special{pa 4314 984}%
\special{fp}%
% BOX 2 5 2 0
% 2 3892 939 4031 1066
% 
\special{pn 8}%
\special{sh 0}%
\special{pa 3892 940}%
\special{pa 4032 940}%
\special{pa 4032 1066}%
\special{pa 3892 1066}%
\special{pa 3892 940}%
\special{ip}%
% ELLIPSE 0 0 3 0
% 4 3760 980 3962 860 4538 980 3127 980
% 
\special{pn 20}%
\special{ar 3760 980 202 120  3.1415927 6.2831853}%
% ELLIPSE 0 0 3 0
% 4 3370 980 3955 1196 2551 1078 5692 980
% 
\special{pn 20}%
\special{ar 3370 980 586 216  6.2831853 6.2831853}%
\special{ar 3370 980 586 216  0.0000000 2.8286589}%
% LINE 2 0 3 0
% 2 2680 1460 4127 1460
% 
\special{pn 8}%
\special{pa 2680 1460}%
\special{pa 4128 1460}%
\special{fp}%
% STR 2 0 3 0
% 3 2020 790 2020 890 2 0
% {\small fusion}
\put(20.2000,-8.9000){\makebox(0,0)[lb]{{\small fusion}}}%
% STR 2 0 3 0
% 3 1880 1300 1880 1400 2 0
% {\small half twist$\times$2}
\put(18.8000,-14.0000){\makebox(0,0)[lb]{{\small half twist$\times$2}}}%
\end{picture}%

%% file: pic24.tex
%WinTpicVersion3.08
\unitlength 0.1in
\begin{picture}( 41.9200,  5.5700)(  4.5000,-15.7300)
% LINE 2 0 3 0
% 2 1453 1389 1760 1389
% 
\special{pn 8}%
\special{pa 1454 1390}%
\special{pa 1760 1390}%
\special{fp}%
% LINE 2 0 3 0
% 2 450 1389 757 1389
% 
\special{pn 8}%
\special{pa 450 1390}%
\special{pa 758 1390}%
\special{fp}%
% STR 2 0 3 0
% 3 527 1304 527 1368 2 0
% $a$
\put(5.2700,-13.6800){\makebox(0,0)[lb]{$a$}}%
% STR 2 0 3 0
% 3 1557 1317 1557 1368 2 0
% $a$
\put(15.5700,-13.6800){\makebox(0,0)[lb]{$a$}}%
% STR 2 0 3 0
% 3 994 1122 994 1186 2 0
% $k-a$
\put(9.9400,-11.8600){\makebox(0,0)[lb]{$k-a$}}%
% ELLIPSE 2 0 3 0
% 4 1105 1394 1448 1573 1551 1394 583 1394
% 
\special{pn 8}%
\special{ar 1106 1394 344 180  3.1415927 6.2831853}%
% ELLIPSE 0 0 3 0
% 4 1105 1394 1448 1215 583 1394 1551 1394
% 
\special{pn 20}%
\special{ar 1106 1394 344 180  6.2831853 6.2831853}%
\special{ar 1106 1394 344 180  0.0000000 3.1415927}%
% STR 2 0 3 0
% 3 1810 1496 1810 1560 2 0
% $=\displaystyle{\frac{\langle a\ k \ k-a\rangle}{\langle a \rangle}}$
\put(18.1000,-15.6000){\makebox(0,0)[lb]{$=\displaystyle{\frac{\langle a\ k \ k-a\rangle}{\langle a \rangle}}$}}%
% LINE 2 0 3 0
% 2 2750 1380 3262 1380
% 
\special{pn 8}%
\special{pa 2750 1380}%
\special{pa 3262 1380}%
\special{fp}%
% STR 2 0 3 0
% 3 2950 1309 2950 1360 2 0
% $a$
\put(29.5000,-13.6000){\makebox(0,0)[lb]{$a$}}%
% STR 2 0 3 0
% 3 3390 1496 3390 1560 2 0
% $=\displaystyle{\frac{(-1)^k}{\langle a\rangle}}$
\put(33.9000,-15.6000){\makebox(0,0)[lb]{$=\displaystyle{\frac{(-1)^k}{\langle a\rangle}}$}}%
% STR 2 0 3 0
% 3 4330 1309 4330 1360 2 0
% $a$
\put(43.3000,-13.6000){\makebox(0,0)[lb]{$a$}}%
% LINE 2 0 3 0
% 2 4130 1380 4642 1380
% 
\special{pn 8}%
\special{pa 4130 1380}%
\special{pa 4642 1380}%
\special{fp}%
\end{picture}%

%% file: pic27.tex
%WinTpicVersion3.08
\unitlength 0.1in
\begin{picture}( 42.9000, 11.6600)( 10.3800,-16.0200)
% ELLIPSE 2 0 3 0
% 4 2160 1270 2904 1902 3088 534 1104 1246
% 
\special{pn 8}%
\special{ar 2160 1270 744 632  3.1681016 5.5323330}%
% LINE 2 0 3 0
% 2 1348 1097 1038 1097
% 
\special{pn 8}%
\special{pa 1348 1098}%
\special{pa 1038 1098}%
\special{fp}%
% STR 2 0 3 0
% 3 1088 1150 1088 1214 2 0
% ext
\put(10.8800,-12.1400){\makebox(0,0)[lb]{ext}}%
% LINE 2 0 3 0
% 2 2068 923 2452 923
% 
\special{pn 8}%
\special{pa 2068 924}%
\special{pa 2452 924}%
\special{fp}%
% LINE 2 0 3 0
% 2 2068 1256 2452 1256
% 
\special{pn 8}%
\special{pa 2068 1256}%
\special{pa 2452 1256}%
\special{fp}%
% ELLIPSE 2 0 3 0
% 4 2657 1089 2926 1339 3073 1140 3310 1172
% 
\special{pn 8}%
\special{ar 2658 1090 270 250  0.1354594 6.2831853}%
\special{ar 2658 1090 270 250  0.0000000 0.1314492}%
% STR 2 0 3 0
% 3 1864 1542 1864 1606 2 0
% int
\put(18.6400,-16.0600){\makebox(0,0)[lb]{int}}%
% BOX 2 5 2 0
% 2 1816 1142 1923 1249
% 
\special{pn 8}%
\special{sh 0}%
\special{pa 1816 1142}%
\special{pa 1924 1142}%
\special{pa 1924 1250}%
\special{pa 1816 1250}%
\special{pa 1816 1142}%
\special{ip}%
% STR 2 0 3 0
% 3 2184 1302 2184 1366 2 0
% ext
\put(21.8400,-13.6600){\makebox(0,0)[lb]{ext}}%
% STR 2 0 3 0
% 3 2160 846 2160 910 2 0
% ext
\put(21.6000,-9.1000){\makebox(0,0)[lb]{ext}}%
% BOX 2 5 2 0
% 2 1592 734 1728 910
% 
\special{pn 8}%
\special{sh 0}%
\special{pa 1592 734}%
\special{pa 1728 734}%
\special{pa 1728 910}%
\special{pa 1592 910}%
\special{pa 1592 734}%
\special{ip}%
% ELLIPSE 0 0 3 0
% 4 1740 1082 2134 1382 2272 1082 2744 1088
% 
\special{pn 20}%
\special{ar 1740 1082 394 300  0.0079680 6.2831853}%
% BOX 2 5 2 0
% 2 1808 1326 1915 1433
% 
\special{pn 8}%
\special{sh 0}%
\special{pa 1808 1326}%
\special{pa 1916 1326}%
\special{pa 1916 1434}%
\special{pa 1808 1434}%
\special{pa 1808 1326}%
\special{ip}%
% ELLIPSE 2 0 3 0
% 4 1768 1118 1898 1602 1499 1794 1929 634
% 
\special{pn 8}%
\special{ar 1768 1118 130 484  5.6039161 6.2831853}%
\special{ar 1768 1118 130 484  0.0000000 2.5467584}%
% STR 2 0 3 0
% 3 2040 542 2040 606 2 0
% ext
\put(20.4000,-6.0600){\makebox(0,0)[lb]{ext}}%
% ELLIPSE 2 0 3 0
% 4 4560 1270 5304 1902 5488 534 3504 1246
% 
\special{pn 8}%
\special{ar 4560 1270 744 632  3.1681016 5.5323330}%
% LINE 2 0 3 0
% 2 3749 1097 3440 1097
% 
\special{pn 8}%
\special{pa 3750 1098}%
\special{pa 3440 1098}%
\special{fp}%
% STR 2 0 3 0
% 3 3488 1150 3488 1214 2 0
% ext
\put(34.8800,-12.1400){\makebox(0,0)[lb]{ext}}%
% LINE 2 0 3 0
% 2 4470 923 4854 923
% 
\special{pn 8}%
\special{pa 4470 924}%
\special{pa 4854 924}%
\special{fp}%
% LINE 2 0 3 0
% 2 4470 1256 4854 1256
% 
\special{pn 8}%
\special{pa 4470 1256}%
\special{pa 4854 1256}%
\special{fp}%
% ELLIPSE 0 0 3 0
% 4 5059 1089 5328 1339 5475 1140 5712 1172
% 
\special{pn 20}%
\special{ar 5060 1090 270 250  0.1354594 6.2831853}%
\special{ar 5060 1090 270 250  0.0000000 0.1314492}%
% STR 2 0 3 0
% 3 4265 1542 4265 1606 2 0
% int
\put(42.6500,-16.0600){\makebox(0,0)[lb]{int}}%
% BOX 2 5 2 0
% 2 4217 1142 4324 1249
% 
\special{pn 8}%
\special{sh 0}%
\special{pa 4218 1142}%
\special{pa 4324 1142}%
\special{pa 4324 1250}%
\special{pa 4218 1250}%
\special{pa 4218 1142}%
\special{ip}%
% STR 2 0 3 0
% 3 4585 1302 4585 1366 2 0
% int
\put(45.8500,-13.6600){\makebox(0,0)[lb]{int}}%
% STR 2 0 3 0
% 3 4561 846 4561 910 2 0
% int
\put(45.6100,-9.1000){\makebox(0,0)[lb]{int}}%
% BOX 2 5 2 0
% 2 3993 734 4129 910
% 
\special{pn 8}%
\special{sh 0}%
\special{pa 3994 734}%
\special{pa 4130 734}%
\special{pa 4130 910}%
\special{pa 3994 910}%
\special{pa 3994 734}%
\special{ip}%
% ELLIPSE 0 0 3 0
% 4 4142 1082 4536 1382 4673 1082 5145 1088
% 
\special{pn 20}%
\special{ar 4142 1082 394 300  0.0079759 6.2831853}%
% BOX 2 5 2 0
% 2 4209 1326 4316 1433
% 
\special{pn 8}%
\special{sh 0}%
\special{pa 4210 1326}%
\special{pa 4316 1326}%
\special{pa 4316 1434}%
\special{pa 4210 1434}%
\special{pa 4210 1326}%
\special{ip}%
% ELLIPSE 2 0 3 0
% 4 4169 1118 4300 1602 3900 1794 4331 634
% 
\special{pn 8}%
\special{ar 4170 1118 132 484  5.6031970 6.2831853}%
\special{ar 4170 1118 132 484  0.0000000 2.5442127}%
% STR 2 0 3 0
% 3 4441 542 4441 606 2 0
% int
\put(44.4100,-6.0600){\makebox(0,0)[lb]{int}}%
\end{picture}%

%% file: pic12.tex
%WinTpicVersion3.08
\unitlength 0.1in
\begin{picture}( 13.4200,  7.8500)( 15.0100,-15.8500)
% ELLIPSE 0 0 3 0
% 4 2172 1208 2756 1528 1740 312 2532 696
% 
\special{pn 20}%
\special{ar 2172 1208 584 320  5.0802801 6.2831853}%
\special{ar 2172 1208 584 320  0.0000000 4.4540720}%
% ELLIPSE 2 0 3 0
% 4 2172 1208 2843 1585 1693 864 1455 690
% 
\special{pn 8}%
\special{ar 2172 1208 672 378  4.0514209 6.2831853}%
\special{ar 2172 1208 672 378  0.0000000 4.0482988}%
% BOX 2 0 2 0
% 2 1900 800 2468 1016
% 
\special{pn 8}%
\special{sh 0}%
\special{pa 1900 800}%
\special{pa 2468 800}%
\special{pa 2468 1016}%
\special{pa 1900 1016}%
\special{pa 1900 800}%
\special{fp}%
% STR 2 0 3 0
% 3 1964 907 1964 984 2 0
% $D(\lambda,j)$
\put(19.6400,-9.8400){\makebox(0,0)[lb]{$D(\lambda,j)$}}%
\end{picture}%

%% file: pic23.tex
%WinTpicVersion3.08
\unitlength 0.1in
\begin{picture}( 26.2000,  6.6000)( 16.4000,-15.3000)
% LINE 2 0 3 0
% 4 3650 1330 4250 1330 3950 1330 3950 1530
% 
\special{pn 8}%
\special{pa 3650 1330}%
\special{pa 4250 1330}%
\special{fp}%
\special{pa 3950 1330}%
\special{pa 3950 1530}%
\special{fp}%
% STR 2 0 3 0
% 3 3660 1250 3660 1300 2 0
% $a$
\put(36.6000,-13.0000){\makebox(0,0)[lb]{$a$}}%
% STR 2 0 3 0
% 3 4160 1250 4160 1300 2 0
% $b$
\put(41.6000,-13.0000){\makebox(0,0)[lb]{$b$}}%
% STR 2 0 3 0
% 3 3990 1520 3990 1570 2 0
% $c$
\put(39.9000,-15.7000){\makebox(0,0)[lb]{$c$}}%
% BOX 0 5 2 0
% 2 3920 1380 4000 1460
% 
\special{pn 20}%
\special{sh 0}%
\special{pa 3920 1380}%
\special{pa 4000 1380}%
\special{pa 4000 1460}%
\special{pa 3920 1460}%
\special{pa 3920 1380}%
\special{ip}%
% LINE 0 0 3 0
% 2 3660 1410 4260 1410
% 
\special{pn 20}%
\special{pa 3660 1410}%
\special{pa 4260 1410}%
\special{fp}%
% LINE 2 0 3 0
% 4 1840 1375 2440 1375 2140 1375 2140 1175
% 
\special{pn 8}%
\special{pa 1840 1376}%
\special{pa 2440 1376}%
\special{fp}%
\special{pa 2140 1376}%
\special{pa 2140 1176}%
\special{fp}%
% STR 2 0 3 0
% 3 1860 1300 1860 1350 2 0
% $a$
\put(18.6000,-13.5000){\makebox(0,0)[lb]{$a$}}%
% STR 2 0 3 0
% 3 2350 1305 2350 1355 2 0
% $b$
\put(23.5000,-13.5500){\makebox(0,0)[lb]{$b$}}%
% LINE 0 0 3 0
% 2 1840 1475 2440 1475
% 
\special{pn 20}%
\special{pa 1840 1476}%
\special{pa 2440 1476}%
\special{fp}%
% STR 2 0 3 0
% 3 2170 1150 2170 1200 2 0
% $c$
\put(21.7000,-12.0000){\makebox(0,0)[lb]{$c$}}%
% STR 2 0 3 0
% 3 1640 990 1640 1040 2 0
% (I)$_{a,b,c}$
\put(16.4000,-10.4000){\makebox(0,0)[lb]{(I)$_{a,b,c}$}}%
% STR 2 0 3 0
% 3 3440 1000 3440 1050 2 0
% (II)$_{a,b,c}$
\put(34.4000,-10.5000){\makebox(0,0)[lb]{(II)$_{a,b,c}$}}%
\end{picture}%

%% file: pic28.tex
%WinTpicVersion3.08
\unitlength 0.1in
\begin{picture}( 14.1800, 11.6600)( 20.4000,-19.9500)
% ELLIPSE 1 0 3 0
% 4 2569 1608 3098 1995 3146 1885 2157 1171
% 
\special{pn 13}%
\special{ar 2570 1608 530 388  4.1083266 6.2831853}%
\special{ar 2570 1608 530 388  0.0000000 0.5811727}%
% ELLIPSE 2 2 3 0
% 4 2569 1608 3098 1995 3146 1885 3902 2259
% 
\special{pn 8}%
\special{ar 2570 1608 530 388  0.5886948 0.6148956}%
\special{ar 2570 1608 530 388  0.6934983 0.7196991}%
\special{ar 2570 1608 530 388  0.7983018 0.8245026}%
\special{ar 2570 1608 530 388  0.9031052 0.9293061}%
\special{ar 2570 1608 530 388  1.0079087 1.0341096}%
\special{ar 2570 1608 530 388  1.1127122 1.1389131}%
\special{ar 2570 1608 530 388  1.2175157 1.2437166}%
\special{ar 2570 1608 530 388  1.3223192 1.3485201}%
\special{ar 2570 1608 530 388  1.4271227 1.4533236}%
\special{ar 2570 1608 530 388  1.5319262 1.5581271}%
\special{ar 2570 1608 530 388  1.6367297 1.6629306}%
\special{ar 2570 1608 530 388  1.7415332 1.7677341}%
\special{ar 2570 1608 530 388  1.8463367 1.8725376}%
\special{ar 2570 1608 530 388  1.9511402 1.9773411}%
\special{ar 2570 1608 530 388  2.0559437 2.0821445}%
\special{ar 2570 1608 530 388  2.1607472 2.1869480}%
\special{ar 2570 1608 530 388  2.2655507 2.2917515}%
\special{ar 2570 1608 530 388  2.3703542 2.3965550}%
\special{ar 2570 1608 530 388  2.4751576 2.5013585}%
\special{ar 2570 1608 530 388  2.5799611 2.6061620}%
\special{ar 2570 1608 530 388  2.6847646 2.7109655}%
\special{ar 2570 1608 530 388  2.7895681 2.8157690}%
\special{ar 2570 1608 530 388  2.8943716 2.9205725}%
\special{ar 2570 1608 530 388  2.9991751 3.0253760}%
\special{ar 2570 1608 530 388  3.1039786 3.1301795}%
\special{ar 2570 1608 530 388  3.2087821 3.2349830}%
\special{ar 2570 1608 530 388  3.3135856 3.3397865}%
\special{ar 2570 1608 530 388  3.4183891 3.4445900}%
\special{ar 2570 1608 530 388  3.5231926 3.5493935}%
\special{ar 2570 1608 530 388  3.6279961 3.6541969}%
\special{ar 2570 1608 530 388  3.7327996 3.7590004}%
\special{ar 2570 1608 530 388  3.8376031 3.8638039}%
\special{ar 2570 1608 530 388  3.9424066 3.9686074}%
\special{ar 2570 1608 530 388  4.0472100 4.0734109}%
\special{ar 2570 1608 530 388  4.1520135 4.1782144}%
\special{ar 2570 1608 530 388  4.2568170 4.2830179}%
\special{ar 2570 1608 530 388  4.3616205 4.3878214}%
\special{ar 2570 1608 530 388  4.4664240 4.4926249}%
\special{ar 2570 1608 530 388  4.5712275 4.5974284}%
\special{ar 2570 1608 530 388  4.6760310 4.7022319}%
\special{ar 2570 1608 530 388  4.7808345 4.8070354}%
\special{ar 2570 1608 530 388  4.8856380 4.9118389}%
\special{ar 2570 1608 530 388  4.9904415 5.0166424}%
\special{ar 2570 1608 530 388  5.0952450 5.1214459}%
\special{ar 2570 1608 530 388  5.2000485 5.2262493}%
\special{ar 2570 1608 530 388  5.3048520 5.3310528}%
\special{ar 2570 1608 530 388  5.4096555 5.4358563}%
\special{ar 2570 1608 530 388  5.5144590 5.5406598}%
\special{ar 2570 1608 530 388  5.6192625 5.6454633}%
\special{ar 2570 1608 530 388  5.7240659 5.7502668}%
\special{ar 2570 1608 530 388  5.8288694 5.8550703}%
\special{ar 2570 1608 530 388  5.9336729 5.9598738}%
\special{ar 2570 1608 530 388  6.0384764 6.0646773}%
\special{ar 2570 1608 530 388  6.1432799 6.1694808}%
\special{ar 2570 1608 530 388  6.2480834 6.2742843}%
\special{ar 2570 1608 530 388  6.3528869 6.3790878}%
\special{ar 2570 1608 530 388  6.4576904 6.4838913}%
\special{ar 2570 1608 530 388  6.5624939 6.5886948}%
\special{ar 2570 1608 530 388  6.6672974 6.6934983}%
\special{ar 2570 1608 530 388  6.7721009 6.7983018}%
% LINE 2 0 3 0
% 2 3090 1657 3458 1657
% 
\special{pn 8}%
\special{pa 3090 1658}%
\special{pa 3458 1658}%
\special{fp}%
% STR 2 0 3 0
% 3 2373 937 2373 999 2 0
% $j_1'$
\put(23.7300,-9.9900){\makebox(0,0)[lb]{$j_1'$}}%
% STR 2 0 3 0
% 3 2838 987 2838 1048 2 0
% $j_2'$
\put(28.3800,-10.4800){\makebox(0,0)[lb]{$j_2'$}}%
% STR 2 0 3 0
% 3 3324 1742 3324 1803 2 0
% $j_n'$
\put(33.2400,-18.0300){\makebox(0,0)[lb]{$j_n'$}}%
% LINE 2 0 3 0
% 2 2426 1232 2298 913
% 
\special{pn 8}%
\special{pa 2426 1232}%
\special{pa 2298 914}%
\special{fp}%
% LINE 2 0 3 0
% 2 2732 1246 2810 933
% 
\special{pn 8}%
\special{pa 2732 1246}%
\special{pa 2810 934}%
\special{fp}%
% STR 2 0 3 0
% 3 2233 1377 2233 1438 2 0
% $j(f_1)$
\put(22.3300,-14.3800){\makebox(0,0)[lb]{$j(f_1)$}}%
% STR 2 0 3 0
% 3 2847 1209 2847 1270 2 0
% $j(f_3)$
\put(28.4700,-12.7000){\makebox(0,0)[lb]{$j(f_3)$}}%
% STR 2 0 3 0
% 3 2470 1141 2470 1201 2 0
% $j(f_2)$
\put(24.7000,-12.0100){\makebox(0,0)[lb]{$j(f_2)$}}%
% STR 2 0 3 0
% 3 2761 1742 2761 1803 2 0
% $j(f_1)$
\put(27.6100,-18.0300){\makebox(0,0)[lb]{$j(f_1)$}}%
% STR 2 0 3 0
% 3 3131 1539 3131 1601 2 0
% $j(f_n)$
\put(31.3100,-16.0100){\makebox(0,0)[lb]{$j(f_n)$}}%
% LINE 2 2 3 0
% 2 3078 1294 3231 1448
% 
\special{pn 8}%
\special{pa 3078 1294}%
\special{pa 3232 1448}%
\special{dt 0.045}%
\end{picture}%

%% file: pic6.tex
%WinTpicVersion3.08
\unitlength 0.1in
\begin{picture}( 36.7400, 13.2200)(  1.5000,-16.1200)
% LINE 2 0 3 0
% 2 800 757 1520 757
% 
\special{pn 8}%
\special{pa 800 758}%
\special{pa 1520 758}%
\special{fp}%
% LINE 2 0 3 0
% 2 1160 757 1160 517
% 
\special{pn 8}%
\special{pa 1160 758}%
\special{pa 1160 518}%
\special{fp}%
% LINE 0 0 3 0
% 2 800 877 1520 877
% 
\special{pn 20}%
\special{pa 800 878}%
\special{pa 1520 878}%
\special{fp}%
% STR 2 0 3 0
% 3 800 577 800 637 2 0
% $a$
\put(8.0000,-6.3700){\makebox(0,0)[lb]{$a$}}%
% STR 2 0 3 0
% 3 1520 577 1520 637 2 0
% $b$
\put(15.2000,-6.3700){\makebox(0,0)[lb]{$b$}}%
% STR 2 0 3 0
% 3 1160 409 1160 469 2 0
% $c$
\put(11.6000,-4.6900){\makebox(0,0)[lb]{$c$}}%
% LINE 2 0 3 0
% 2 3092 751 3092 391
% 
\special{pn 8}%
\special{pa 3092 752}%
\special{pa 3092 392}%
\special{fp}%
% LINE 0 0 3 0
% 2 2372 991 2612 751
% 
\special{pn 20}%
\special{pa 2372 992}%
\special{pa 2612 752}%
\special{fp}%
% LINE 0 0 3 0
% 2 3824 991 3584 751
% 
\special{pn 20}%
\special{pa 3824 992}%
\special{pa 3584 752}%
\special{fp}%
% LINE 2 0 3 0
% 2 2372 751 3812 751
% 
\special{pn 8}%
\special{pa 2372 752}%
\special{pa 3812 752}%
\special{fp}%
% ELLIPSE 0 0 3 0
% 4 3092 757 3278 865 2528 757 3992 757
% 
\special{pn 20}%
\special{ar 3092 758 186 108  6.2831853 6.2831853}%
\special{ar 3092 758 186 108  0.0000000 3.1415927}%
% STR 2 0 3 0
% 3 2390 625 2390 685 2 0
% $a$
\put(23.9000,-6.8500){\makebox(0,0)[lb]{$a$}}%
% STR 2 0 3 0
% 3 3722 631 3722 691 2 0
% $b$
\put(37.2200,-6.9100){\makebox(0,0)[lb]{$b$}}%
% STR 2 0 3 0
% 3 3120 400 3120 460 2 0
% $c$
\put(31.2000,-4.6000){\makebox(0,0)[lb]{$c$}}%
% STR 2 0 3 0
% 3 2906 625 2906 685 2 0
% $a$
\put(29.0600,-6.8500){\makebox(0,0)[lb]{$a$}}%
% STR 2 0 3 0
% 3 3152 631 3152 691 2 0
% $b$
\put(31.5200,-6.9100){\makebox(0,0)[lb]{$b$}}%
% STR 2 0 3 0
% 3 2600 880 2600 940 2 0
% $k-a$
\put(26.0000,-9.4000){\makebox(0,0)[lb]{$k-a$}}%
% STR 2 0 3 0
% 3 3296 892 3296 952 2 0
% $k-b$
\put(32.9600,-9.5200){\makebox(0,0)[lb]{$k-b$}}%
% LINE 2 0 3 0
% 2 3092 1372 3090 1120
% 
\special{pn 8}%
\special{pa 3092 1372}%
\special{pa 3090 1120}%
\special{fp}%
% LINE 0 0 3 0
% 2 2372 1612 2612 1372
% 
\special{pn 20}%
\special{pa 2372 1612}%
\special{pa 2612 1372}%
\special{fp}%
% LINE 0 0 3 0
% 2 3824 1612 3584 1372
% 
\special{pn 20}%
\special{pa 3824 1612}%
\special{pa 3584 1372}%
\special{fp}%
% LINE 2 0 3 0
% 2 2372 1372 3812 1372
% 
\special{pn 8}%
\special{pa 2372 1372}%
\special{pa 3812 1372}%
\special{fp}%
% STR 2 0 3 0
% 3 2390 1246 2390 1306 2 0
% $a$
\put(23.9000,-13.0600){\makebox(0,0)[lb]{$a$}}%
% STR 2 0 3 0
% 3 3722 1252 3722 1312 2 0
% $b$
\put(37.2200,-13.1200){\makebox(0,0)[lb]{$b$}}%
% STR 2 0 3 0
% 3 3130 1120 3130 1180 2 0
% $c$
\put(31.3000,-11.8000){\makebox(0,0)[lb]{$c$}}%
% STR 2 0 3 0
% 3 2690 1420 2690 1480 2 0
% $k-a$
\put(26.9000,-14.8000){\makebox(0,0)[lb]{$k-a$}}%
% STR 2 0 3 0
% 3 3212 1420 3212 1480 2 0
% $k-b$
\put(32.1200,-14.8000){\makebox(0,0)[lb]{$k-b$}}%
% VECTOR 1 0 3 0
% 2 1760 757 2120 757
% 
\special{pn 13}%
\special{pa 1760 758}%
\special{pa 2120 758}%
\special{fp}%
\special{sh 1}%
\special{pa 2120 758}%
\special{pa 2054 738}%
\special{pa 2068 758}%
\special{pa 2054 778}%
\special{pa 2120 758}%
\special{fp}%
% VECTOR 1 0 3 0
% 2 1760 1384 2120 1384
% 
\special{pn 13}%
\special{pa 1760 1384}%
\special{pa 2120 1384}%
\special{fp}%
\special{sh 1}%
\special{pa 2120 1384}%
\special{pa 2054 1364}%
\special{pa 2068 1384}%
\special{pa 2054 1404}%
\special{pa 2120 1384}%
\special{fp}%
% STR 2 0 3 0
% 3 150 660 150 760 2 0
% (I)$_{a,b,c} \ ;$ 
\put(1.5000,-7.6000){\makebox(0,0)[lb]{(I)$_{a,b,c} \ ;$ }}%
% STR 2 0 3 0
% 3 1690 610 1690 710 2 0
% {\small fusion$\times$2}
\put(16.9000,-7.1000){\makebox(0,0)[lb]{{\small fusion$\times$2}}}%
% STR 2 0 3 0
% 3 1620 1230 1620 1330 2 0
% {\small tetrahedron} 
\put(16.2000,-13.3000){\makebox(0,0)[lb]{{\small tetrahedron} }}%
\end{picture}%

%% file: pic9.tex
%WinTpicVersion3.08
\unitlength 0.1in
\begin{picture}( 48.0000, 21.1000)(  0.0000,-26.2000)
% STR 2 0 3 0
% 3 0 670 0 720 2 0
% (II)$_{a,b,c} \ ;$
\put(0.0000,-7.2000){\makebox(0,0)[lb]{(II)$_{a,b,c} \ ;$}}%
% VECTOR 1 0 3 0
% 2 1800 720 2200 720
% 
\special{pn 13}%
\special{pa 1800 720}%
\special{pa 2200 720}%
\special{fp}%
\special{sh 1}%
\special{pa 2200 720}%
\special{pa 2134 700}%
\special{pa 2148 720}%
\special{pa 2134 740}%
\special{pa 2200 720}%
\special{fp}%
% STR 2 0 3 0
% 3 2460 720 2460 770 2 0
% $a$
\put(24.6000,-7.7000){\makebox(0,0)[lb]{$a$}}%
% STR 2 0 3 0
% 3 4550 730 4550 780 2 0
% $b$
\put(45.5000,-7.8000){\makebox(0,0)[lb]{$b$}}%
% LINE 2 0 3 0
% 2 2380 800 4800 800
% 
\special{pn 8}%
\special{pa 2380 800}%
\special{pa 4800 800}%
\special{fp}%
% LINE 2 0 3 0
% 2 3527 800 3530 1310
% 
\special{pn 8}%
\special{pa 3528 800}%
\special{pa 3530 1310}%
\special{fp}%
% LINE 0 0 3 0
% 2 2360 1010 2750 800
% 
\special{pn 20}%
\special{pa 2360 1010}%
\special{pa 2750 800}%
\special{fp}%
% LINE 0 0 3 0
% 2 4770 1010 4380 800
% 
\special{pn 20}%
\special{pa 4770 1010}%
\special{pa 4380 800}%
\special{fp}%
% BOX 0 5 2 0
% 2 3470 1130 3596 1004
% 
\special{pn 20}%
\special{sh 0}%
\special{pa 3470 1130}%
\special{pa 3596 1130}%
\special{pa 3596 1004}%
\special{pa 3470 1004}%
\special{pa 3470 1130}%
\special{ip}%
% ELLIPSE 0 0 3 0
% 4 3530 810 3887 1094 3236 810 4045 810
% 
\special{pn 20}%
\special{ar 3530 810 358 284  6.2831853 6.2831853}%
\special{ar 3530 810 358 284  0.0000000 3.1415927}%
% STR 2 0 3 0
% 3 3610 1320 3610 1370 2 0
% $c$
\put(36.1000,-13.7000){\makebox(0,0)[lb]{$c$}}%
% STR 2 0 3 0
% 3 2800 900 2800 950 2 0
% $k-a$
\put(28.0000,-9.5000){\makebox(0,0)[lb]{$k-a$}}%
% STR 2 0 3 0
% 3 3970 900 3970 950 2 0
% $k-b$
\put(39.7000,-9.5000){\makebox(0,0)[lb]{$k-b$}}%
% STR 2 0 3 0
% 3 3290 720 3290 770 2 0
% $a$
\put(32.9000,-7.7000){\makebox(0,0)[lb]{$a$}}%
% STR 2 0 3 0
% 3 3700 730 3700 780 2 0
% $b$
\put(37.0000,-7.8000){\makebox(0,0)[lb]{$b$}}%
% VECTOR 1 0 3 0
% 2 1800 1920 2200 1920
% 
\special{pn 13}%
\special{pa 1800 1920}%
\special{pa 2200 1920}%
\special{fp}%
\special{sh 1}%
\special{pa 2200 1920}%
\special{pa 2134 1900}%
\special{pa 2148 1920}%
\special{pa 2134 1940}%
\special{pa 2200 1920}%
\special{fp}%
% STR 2 0 3 0
% 3 2420 1540 2420 1590 2 0
% $a$
\put(24.2000,-15.9000){\makebox(0,0)[lb]{$a$}}%
% STR 2 0 3 0
% 3 4540 1550 4540 1600 2 0
% $b$
\put(45.4000,-16.0000){\makebox(0,0)[lb]{$b$}}%
% LINE 2 0 3 0
% 2 2380 1620 4800 1620
% 
\special{pn 8}%
\special{pa 2380 1620}%
\special{pa 4800 1620}%
\special{fp}%
% LINE 0 0 3 0
% 2 2360 1830 2750 1620
% 
\special{pn 20}%
\special{pa 2360 1830}%
\special{pa 2750 1620}%
\special{fp}%
% LINE 0 0 3 0
% 2 4770 1830 4380 1620
% 
\special{pn 20}%
\special{pa 4770 1830}%
\special{pa 4380 1620}%
\special{fp}%
% STR 2 0 3 0
% 3 3610 2550 3610 2600 2 0
% $c$
\put(36.1000,-26.0000){\makebox(0,0)[lb]{$c$}}%
% STR 2 0 3 0
% 3 2830 1710 2830 1760 2 0
% $k-a$
\put(28.3000,-17.6000){\makebox(0,0)[lb]{$k-a$}}%
% STR 2 0 3 0
% 3 3970 1710 3970 1760 2 0
% $k-b$
\put(39.7000,-17.6000){\makebox(0,0)[lb]{$k-b$}}%
% STR 2 0 3 0
% 3 3320 1540 3320 1590 2 0
% $a$
\put(33.2000,-15.9000){\makebox(0,0)[lb]{$a$}}%
% LINE 0 0 3 0
% 2 3210 1620 3520 1840
% 
\special{pn 20}%
\special{pa 3210 1620}%
\special{pa 3520 1840}%
\special{fp}%
% LINE 2 0 3 0
% 2 3520 1620 3520 2620
% 
\special{pn 8}%
\special{pa 3520 1620}%
\special{pa 3520 2620}%
\special{fp}%
% LINE 0 0 3 0
% 2 3670 2300 3950 1620
% 
\special{pn 20}%
\special{pa 3670 2300}%
\special{pa 3950 1620}%
\special{fp}%
% STR 2 0 3 0
% 3 3180 1950 3180 2000 2 0
% $k-c$
\put(31.8000,-20.0000){\makebox(0,0)[lb]{$k-c$}}%
% STR 2 0 3 0
% 3 3540 1710 3540 1760 2 0
% $c$
\put(35.4000,-17.6000){\makebox(0,0)[lb]{$c$}}%
% STR 2 0 3 0
% 3 3670 1530 3670 1580 2 0
% $b$
\put(36.7000,-15.8000){\makebox(0,0)[lb]{$b$}}%
% STR 2 0 3 0
% 3 1720 590 1720 690 2 0
% {\small fusion$\times$2}
\put(17.2000,-6.9000){\makebox(0,0)[lb]{{\small fusion$\times$2}}}%
% STR 2 0 3 0
% 3 1790 1780 1790 1880 2 0
% {\small fusion}
\put(17.9000,-18.8000){\makebox(0,0)[lb]{{\small fusion}}}%
% LINE 2 0 3 0
% 4 820 710 1420 710 1120 710 1120 910
% 
\special{pn 8}%
\special{pa 820 710}%
\special{pa 1420 710}%
\special{fp}%
\special{pa 1120 710}%
\special{pa 1120 910}%
\special{fp}%
% STR 2 0 3 0
% 3 830 630 830 680 2 0
% $a$
\put(8.3000,-6.8000){\makebox(0,0)[lb]{$a$}}%
% STR 2 0 3 0
% 3 1330 630 1330 680 2 0
% $b$
\put(13.3000,-6.8000){\makebox(0,0)[lb]{$b$}}%
% STR 2 0 3 0
% 3 1160 900 1160 950 2 0
% $c$
\put(11.6000,-9.5000){\makebox(0,0)[lb]{$c$}}%
% BOX 0 5 2 0
% 2 1090 760 1170 840
% 
\special{pn 20}%
\special{sh 0}%
\special{pa 1090 760}%
\special{pa 1170 760}%
\special{pa 1170 840}%
\special{pa 1090 840}%
\special{pa 1090 760}%
\special{ip}%
% LINE 0 0 3 0
% 2 830 790 1430 790
% 
\special{pn 20}%
\special{pa 830 790}%
\special{pa 1430 790}%
\special{fp}%
% BOX 0 5 2 0
% 2 3470 2500 3596 2374
% 
\special{pn 20}%
\special{sh 0}%
\special{pa 3470 2500}%
\special{pa 3596 2500}%
\special{pa 3596 2374}%
\special{pa 3470 2374}%
\special{pa 3470 2500}%
\special{ip}%
% ELLIPSE 0 0 3 0
% 4 3520 2230 3360 2420 3510 2010 3840 2360
% 
\special{pn 20}%
\special{ar 3520 2230 160 190  0.3282986 4.6583875}%
% STR 2 0 3 0
% 3 3670 1530 3670 1580 2 0
% $b$
\put(36.7000,-15.8000){\makebox(0,0)[lb]{$b$}}%
\end{picture}%

%% file: pic9_.tex
%WinTpicVersion3.08
\unitlength 0.1in
\begin{picture}( 43.8000, 22.5000)( 16.2000,-25.6000)
% VECTOR 1 0 3 0
% 2 3000 1620 3400 1620
% 
\special{pn 13}%
\special{pa 3000 1620}%
\special{pa 3400 1620}%
\special{fp}%
\special{sh 1}%
\special{pa 3400 1620}%
\special{pa 3334 1600}%
\special{pa 3348 1620}%
\special{pa 3334 1640}%
\special{pa 3400 1620}%
\special{fp}%
% STR 2 0 3 0
% 3 3610 1430 3610 1480 2 0
% $a$
\put(36.1000,-14.8000){\makebox(0,0)[lb]{$a$}}%
% STR 2 0 3 0
% 3 5780 1450 5780 1500 2 0
% $b$
\put(57.8000,-15.0000){\makebox(0,0)[lb]{$b$}}%
% LINE 2 0 3 0
% 2 3580 1520 6000 1520
% 
\special{pn 8}%
\special{pa 3580 1520}%
\special{pa 6000 1520}%
\special{fp}%
% LINE 0 0 3 0
% 2 3560 1730 3950 1520
% 
\special{pn 20}%
\special{pa 3560 1730}%
\special{pa 3950 1520}%
\special{fp}%
% LINE 0 0 3 0
% 2 5970 1730 5580 1520
% 
\special{pn 20}%
\special{pa 5970 1730}%
\special{pa 5580 1520}%
\special{fp}%
% STR 2 0 3 0
% 3 4720 1990 4720 2040 2 0
% $c$
\put(47.2000,-20.4000){\makebox(0,0)[lb]{$c$}}%
% STR 2 0 3 0
% 3 4160 1430 4160 1480 2 0
% $k-a$
\put(41.6000,-14.8000){\makebox(0,0)[lb]{$k-a$}}%
% STR 2 0 3 0
% 3 5250 1620 5250 1670 2 0
% $k-b$
\put(52.5000,-16.7000){\makebox(0,0)[lb]{$k-b$}}%
% STR 2 0 3 0
% 3 4350 1680 4350 1730 2 0
% $k-c$
\put(43.5000,-17.3000){\makebox(0,0)[lb]{$k-c$}}%
% STR 2 0 3 0
% 3 4820 1440 4820 1490 2 0
% $b$
\put(48.2000,-14.9000){\makebox(0,0)[lb]{$b$}}%
% LINE 2 0 3 0
% 2 4700 1520 4700 2120
% 
\special{pn 8}%
\special{pa 4700 1520}%
\special{pa 4700 2120}%
\special{fp}%
% LINE 0 0 3 0
% 2 4700 1850 5200 1520
% 
\special{pn 20}%
\special{pa 4700 1850}%
\special{pa 5200 1520}%
\special{fp}%
% VECTOR 1 0 3 0
% 2 3010 2400 3410 2400
% 
\special{pn 13}%
\special{pa 3010 2400}%
\special{pa 3410 2400}%
\special{fp}%
\special{sh 1}%
\special{pa 3410 2400}%
\special{pa 3344 2380}%
\special{pa 3358 2400}%
\special{pa 3344 2420}%
\special{pa 3410 2400}%
\special{fp}%
% STR 2 0 3 0
% 3 3630 2260 3630 2310 2 0
% $a$
\put(36.3000,-23.1000){\makebox(0,0)[lb]{$a$}}%
% STR 2 0 3 0
% 3 5770 2260 5770 2310 2 0
% $b$
\put(57.7000,-23.1000){\makebox(0,0)[lb]{$b$}}%
% LINE 2 0 3 0
% 2 3580 2320 6000 2320
% 
\special{pn 8}%
\special{pa 3580 2320}%
\special{pa 6000 2320}%
\special{fp}%
% LINE 0 0 3 0
% 2 3640 2530 4030 2320
% 
\special{pn 20}%
\special{pa 3640 2530}%
\special{pa 4030 2320}%
\special{fp}%
% LINE 0 0 3 0
% 2 5970 2530 5580 2320
% 
\special{pn 20}%
\special{pa 5970 2530}%
\special{pa 5580 2320}%
\special{fp}%
% STR 2 0 3 0
% 3 4740 2500 4740 2550 2 0
% $c$
\put(47.4000,-25.5000){\makebox(0,0)[lb]{$c$}}%
% STR 2 0 3 0
% 3 4200 2260 4200 2310 2 0
% $k-a$
\put(42.0000,-23.1000){\makebox(0,0)[lb]{$k-a$}}%
% STR 2 0 3 0
% 3 5060 2270 5060 2320 2 0
% $k-b$
\put(50.6000,-23.2000){\makebox(0,0)[lb]{$k-b$}}%
% LINE 2 0 3 0
% 2 4700 2320 4700 2560
% 
\special{pn 8}%
\special{pa 4700 2320}%
\special{pa 4700 2560}%
\special{fp}%
% STR 2 0 3 0
% 3 2870 1470 2870 1570 2 0
% {\small tetrahedron}
\put(28.7000,-15.7000){\makebox(0,0)[lb]{{\small tetrahedron}}}%
% STR 2 0 3 0
% 3 2890 2240 2890 2340 2 0
% {\small tetrahedron}
\put(28.9000,-23.4000){\makebox(0,0)[lb]{{\small tetrahedron}}}%
% LINE 2 5 3 0
% 4 1620 710 2220 710 1920 710 1920 910
% 
\special{pn 8}%
\special{pa 1620 710}%
\special{pa 2220 710}%
\special{ip}%
\special{pa 1920 710}%
\special{pa 1920 910}%
\special{ip}%
% LINE 2 5 3 0
% 2 1630 790 2230 790
% 
\special{pn 8}%
\special{pa 1630 790}%
\special{pa 2230 790}%
\special{ip}%
% VECTOR 1 0 3 0
% 2 3020 640 3420 640
% 
\special{pn 13}%
\special{pa 3020 640}%
\special{pa 3420 640}%
\special{fp}%
\special{sh 1}%
\special{pa 3420 640}%
\special{pa 3354 620}%
\special{pa 3368 640}%
\special{pa 3354 660}%
\special{pa 3420 640}%
\special{fp}%
% STR 2 0 3 0
% 3 3610 440 3610 490 2 0
% $a$
\put(36.1000,-4.9000){\makebox(0,0)[lb]{$a$}}%
% STR 2 0 3 0
% 3 5730 450 5730 500 2 0
% $b$
\put(57.3000,-5.0000){\makebox(0,0)[lb]{$b$}}%
% LINE 2 0 3 0
% 2 3580 510 6000 510
% 
\special{pn 8}%
\special{pa 3580 510}%
\special{pa 6000 510}%
\special{fp}%
% LINE 0 0 3 0
% 2 3560 720 3950 510
% 
\special{pn 20}%
\special{pa 3560 720}%
\special{pa 3950 510}%
\special{fp}%
% LINE 0 0 3 0
% 2 5970 720 5580 510
% 
\special{pn 20}%
\special{pa 5970 720}%
\special{pa 5580 510}%
\special{fp}%
% STR 2 0 3 0
% 3 4750 1230 4750 1280 2 0
% $c$
\put(47.5000,-12.8000){\makebox(0,0)[lb]{$c$}}%
% STR 2 0 3 0
% 3 4000 620 4000 670 2 0
% $k-a$
\put(40.0000,-6.7000){\makebox(0,0)[lb]{$k-a$}}%
% STR 2 0 3 0
% 3 5260 600 5260 650 2 0
% $k-b$
\put(52.6000,-6.5000){\makebox(0,0)[lb]{$k-b$}}%
% STR 2 0 3 0
% 3 4570 430 4570 480 2 0
% $a$
\put(45.7000,-4.8000){\makebox(0,0)[lb]{$a$}}%
% LINE 0 0 3 0
% 2 4410 510 4720 730
% 
\special{pn 20}%
\special{pa 4410 510}%
\special{pa 4720 730}%
\special{fp}%
% LINE 0 0 3 0
% 2 4730 1050 5230 510
% 
\special{pn 20}%
\special{pa 4730 1050}%
\special{pa 5230 510}%
\special{fp}%
% STR 2 0 3 0
% 3 4370 850 4370 900 2 0
% $k-c$
\put(43.7000,-9.0000){\makebox(0,0)[lb]{$k-c$}}%
% STR 2 0 3 0
% 3 4770 610 4770 660 2 0
% $c$
\put(47.7000,-6.6000){\makebox(0,0)[lb]{$c$}}%
% STR 2 0 3 0
% 3 4960 440 4960 490 2 0
% $b$
\put(49.6000,-4.9000){\makebox(0,0)[lb]{$b$}}%
% STR 2 0 3 0
% 3 2940 480 2940 580 2 0
% {\small half twist}
\put(29.4000,-5.8000){\makebox(0,0)[lb]{{\small half twist}}}%
% LINE 2 0 3 0
% 4 4730 1110 4730 510 4730 510 4730 1310
% 
\special{pn 8}%
\special{pa 4730 1110}%
\special{pa 4730 510}%
\special{fp}%
\special{pa 4730 510}%
\special{pa 4730 1310}%
\special{fp}%
\end{picture}%

%% file: pic20.tex
%WinTpicVersion3.08
\unitlength 0.1in
\begin{picture}( 13.5200, 10.0000)(  7.5000,-14.7800)
% ELLIPSE 0 0 3 0
% 4 1482 945 1892 1258 2035 945 2527 950
% 
\special{pn 20}%
\special{ar 1482 946 410 314  0.0066985 6.2831853}%
% LINE 2 0 3 0
% 2 1072 960 750 960
% 
\special{pn 8}%
\special{pa 1072 960}%
\special{pa 750 960}%
\special{fp}%
% LINE 2 0 3 0
% 2 1840 786 2102 653
% 
\special{pn 8}%
\special{pa 1840 786}%
\special{pa 2102 654}%
\special{fp}%
% LINE 2 0 3 0
% 2 1836 1124 2096 1258
% 
\special{pn 8}%
\special{pa 1836 1124}%
\special{pa 2096 1258}%
\special{fp}%
% STR 2 0 3 0
% 3 842 1032 842 1083 2 0
% 1
\put(8.4200,-10.8300){\makebox(0,0)[lb]{1}}%
% STR 2 0 3 0
% 3 2045 730 2045 781 2 0
% 2
\put(20.4500,-7.8100){\makebox(0,0)[lb]{2}}%
% STR 2 0 3 0
% 3 2045 1119 2045 1170 2 0
% 3
\put(20.4500,-11.7000){\makebox(0,0)[lb]{3}}%
% STR 2 0 3 0
% 3 1144 597 1144 648 2 0
% 4
\put(11.4400,-6.4800){\makebox(0,0)[lb]{4}}%
% STR 2 0 3 0
% 3 1134 1247 1134 1298 2 0
% 5
\put(11.3400,-12.9800){\makebox(0,0)[lb]{5}}%
% STR 2 0 3 0
% 3 1494 979 1494 1030 2 0
% 6
\put(14.9400,-10.3000){\makebox(0,0)[lb]{6}}%
% STR 2 0 3 0
% 3 1759 597 1759 648 2 0
% 7
\put(17.5900,-6.4800){\makebox(0,0)[lb]{7}}%
% STR 2 0 3 0
% 3 1753 950 1753 1002 2 0
% 8
\put(17.5300,-10.0200){\makebox(0,0)[lb]{8}}%
% STR 2 0 3 0
% 3 1748 1267 1748 1318 2 0
% 9
\put(17.4800,-13.1800){\makebox(0,0)[lb]{9}}%
% BOX 2 5 2 0
% 2 1518 1198 1630 1310
% 
\special{pn 8}%
\special{sh 0}%
\special{pa 1518 1198}%
\special{pa 1630 1198}%
\special{pa 1630 1310}%
\special{pa 1518 1310}%
\special{pa 1518 1198}%
\special{ip}%
% ELLIPSE 2 0 3 0
% 4 1470 974 1606 1478 1190 1678 1638 470
% 
\special{pn 8}%
\special{ar 1470 974 136 504  5.6026641 6.2831853}%
\special{ar 1470 974 136 504  0.0000000 2.5453935}%
\end{picture}%

%% file: pic19.tex
%WinTpicVersion3.08
\unitlength 0.1in
\begin{picture}( 50.8800, 21.1300)(  1.6000,-28.2300)
% ELLIPSE 0 0 3 0
% 4 809 1290 1185 972 1270 1304 1454 1316
% 
\special{pn 20}%
\special{ar 810 1290 376 318  0.0480251 6.2831853}%
\special{ar 810 1290 376 318  0.0000000 0.0368597}%
% ELLIPSE 2 0 3 0
% 4 1914 1290 2290 972 2375 1304 2560 1316
% 
\special{pn 8}%
\special{ar 1914 1290 376 318  0.0479508 6.2831853}%
\special{ar 1914 1290 376 318  0.0000000 0.0368597}%
% LINE 2 0 3 0
% 2 1101 1101 1617 1101
% 
\special{pn 8}%
\special{pa 1102 1102}%
\special{pa 1618 1102}%
\special{fp}%
% LINE 2 0 3 0
% 2 1122 1482 1614 1482
% 
\special{pn 8}%
\special{pa 1122 1482}%
\special{pa 1614 1482}%
\special{fp}%
% LINE 2 0 3 0
% 2 435 1286 189 1286
% 
\special{pn 8}%
\special{pa 436 1286}%
\special{pa 190 1286}%
\special{fp}%
% LINE 2 0 3 0
% 2 803 978 803 1598
% 
\special{pn 8}%
\special{pa 804 978}%
\special{pa 804 1598}%
\special{fp}%
% LINE 2 0 3 0
% 2 2138 1034 2316 812
% 
\special{pn 8}%
\special{pa 2138 1034}%
\special{pa 2316 812}%
\special{fp}%
% LINE 2 0 3 0
% 2 2118 1556 2290 1715
% 
\special{pn 8}%
\special{pa 2118 1556}%
\special{pa 2290 1716}%
\special{fp}%
% STR 2 0 3 0
% 3 2166 1750 2166 1811 2 0
% 3
\put(21.6600,-18.1100){\makebox(0,0)[lb]{3}}%
% STR 2 0 3 0
% 3 2157 819 2157 880 2 0
% 2
\put(21.5700,-8.8000){\makebox(0,0)[lb]{2}}%
% STR 2 0 3 0
% 3 160 1001 160 1062 2 0
% \fbox{$\Gamma$}
\put(1.6000,-10.6200){\makebox(0,0)[lb]{\fbox{$\Gamma$}}}%
% ELLIPSE 0 0 3 0
% 4 3574 1293 3950 974 4035 1305 4218 1318
% 
\special{pn 20}%
\special{ar 3574 1294 376 320  0.0450007 6.2831853}%
\special{ar 3574 1294 376 320  0.0000000 0.0303594}%
% ELLIPSE 2 0 3 0
% 4 4679 1293 5055 974 5140 1305 5325 1318
% 
\special{pn 8}%
\special{ar 4680 1294 376 320  0.0448615 6.2831853}%
\special{ar 4680 1294 376 320  0.0000000 0.0303594}%
% LINE 2 0 3 0
% 2 3856 1101 4372 1101
% 
\special{pn 8}%
\special{pa 3856 1102}%
\special{pa 4372 1102}%
\special{fp}%
% LINE 2 0 3 0
% 2 3887 1482 4379 1482
% 
\special{pn 8}%
\special{pa 3888 1482}%
\special{pa 4380 1482}%
\special{fp}%
% LINE 2 0 3 0
% 2 3200 1287 2954 1287
% 
\special{pn 8}%
\special{pa 3200 1288}%
\special{pa 2954 1288}%
\special{fp}%
% LINE 2 0 3 0
% 2 3568 980 3568 1600
% 
\special{pn 8}%
\special{pa 3568 980}%
\special{pa 3568 1600}%
\special{fp}%
% LINE 2 0 3 0
% 2 4902 1024 5081 803
% 
\special{pn 8}%
\special{pa 4902 1024}%
\special{pa 5082 804}%
\special{fp}%
% LINE 2 0 3 0
% 2 4883 1557 5055 1718
% 
\special{pn 8}%
\special{pa 4884 1558}%
\special{pa 5056 1718}%
\special{fp}%
% BOX 2 5 2 0
% 2 4103 1011 4214 1606
% 
\special{pn 8}%
\special{sh 0}%
\special{pa 4104 1012}%
\special{pa 4214 1012}%
\special{pa 4214 1606}%
\special{pa 4104 1606}%
\special{pa 4104 1012}%
\special{ip}%
% LINE 2 0 3 0
% 2 1535 1290 2290 1290
% 
\special{pn 8}%
\special{pa 1536 1290}%
\special{pa 2290 1290}%
\special{fp}%
% LINE 2 0 3 0
% 2 4300 1293 5055 1293
% 
\special{pn 8}%
\special{pa 4300 1294}%
\special{pa 5056 1294}%
\special{fp}%
% ELLIPSE 0 0 3 0
% 4 3574 2399 3950 2079 4035 2411 4218 2423
% 
\special{pn 20}%
\special{ar 3574 2400 376 320  0.0434509 6.2831853}%
\special{ar 3574 2400 376 320  0.0000000 0.0303594}%
% ELLIPSE 2 0 3 0
% 4 4679 2399 5055 2079 5140 2411 5325 2423
% 
\special{pn 8}%
\special{ar 4680 2400 376 320  0.0433165 6.2831853}%
\special{ar 4680 2400 376 320  0.0000000 0.0303594}%
% LINE 2 0 3 0
% 2 3875 2195 4391 2195
% 
\special{pn 8}%
\special{pa 3876 2196}%
\special{pa 4392 2196}%
\special{fp}%
% LINE 2 0 3 0
% 2 3887 2589 4379 2589
% 
\special{pn 8}%
\special{pa 3888 2590}%
\special{pa 4380 2590}%
\special{fp}%
% LINE 2 0 3 0
% 2 3200 2393 2954 2393
% 
\special{pn 8}%
\special{pa 3200 2394}%
\special{pa 2954 2394}%
\special{fp}%
% LINE 2 0 3 0
% 2 3568 2086 3568 2706
% 
\special{pn 8}%
\special{pa 3568 2086}%
\special{pa 3568 2706}%
\special{fp}%
% LINE 2 0 3 0
% 2 4902 2138 5081 1917
% 
\special{pn 8}%
\special{pa 4902 2138}%
\special{pa 5082 1918}%
\special{fp}%
% LINE 2 0 3 0
% 2 4883 2663 5055 2823
% 
\special{pn 8}%
\special{pa 4884 2664}%
\special{pa 5056 2824}%
\special{fp}%
% BOX 2 5 2 0
% 2 4103 2116 4214 2711
% 
\special{pn 8}%
\special{sh 0}%
\special{pa 4104 2116}%
\special{pa 4214 2116}%
\special{pa 4214 2712}%
\special{pa 4104 2712}%
\special{pa 4104 2116}%
\special{ip}%
% LINE 2 0 3 0
% 2 4300 2399 5055 2399
% 
\special{pn 8}%
\special{pa 4300 2400}%
\special{pa 5056 2400}%
\special{fp}%
% BOX 2 5 2 0
% 2 4625 2356 4754 2466
% 
\special{pn 8}%
\special{sh 0}%
\special{pa 4626 2356}%
\special{pa 4754 2356}%
\special{pa 4754 2466}%
\special{pa 4626 2466}%
\special{pa 4626 2356}%
\special{ip}%
% VECTOR 0 0 3 0
% 2 2500 1293 2807 1293
% 
\special{pn 20}%
\special{pa 2500 1294}%
\special{pa 2808 1294}%
\special{fp}%
\special{sh 1}%
\special{pa 2808 1294}%
\special{pa 2740 1274}%
\special{pa 2754 1294}%
\special{pa 2740 1314}%
\special{pa 2808 1294}%
\special{fp}%
% VECTOR 0 0 3 0
% 2 2500 2368 2807 2368
% 
\special{pn 20}%
\special{pa 2500 2368}%
\special{pa 2808 2368}%
\special{fp}%
\special{sh 1}%
\special{pa 2808 2368}%
\special{pa 2740 2348}%
\special{pa 2754 2368}%
\special{pa 2740 2388}%
\special{pa 2808 2368}%
\special{fp}%
% STR 2 0 3 0
% 3 448 938 448 1034 2 0
% 4
\put(4.4800,-10.3400){\makebox(0,0)[lb]{4}}%
% STR 2 0 3 0
% 3 458 1619 458 1715 2 0
% 5
\put(4.5800,-17.1500){\makebox(0,0)[lb]{5}}%
% STR 2 0 3 0
% 3 698 1264 698 1360 2 0
% 6
\put(6.9800,-13.6000){\makebox(0,0)[lb]{6}}%
% STR 2 0 3 0
% 3 1024 909 1024 1005 2 0
% 7
\put(10.2400,-10.0500){\makebox(0,0)[lb]{7}}%
% STR 2 0 3 0
% 3 1034 1254 1034 1350 2 0
% 8
\put(10.3400,-13.5000){\makebox(0,0)[lb]{8}}%
% STR 2 0 3 0
% 3 1014 1658 1014 1754 2 0
% 9
\put(10.1400,-17.5400){\makebox(0,0)[lb]{9}}%
% STR 2 0 3 0
% 3 1274 986 1274 1082 2 0
% 10
\put(12.7400,-10.8200){\makebox(0,0)[lb]{10}}%
% STR 2 0 3 0
% 3 1283 1542 1283 1638 2 0
% 11
\put(12.8300,-16.3800){\makebox(0,0)[lb]{11}}%
% STR 2 0 3 0
% 3 1802 861 1802 957 2 0
% 12
\put(18.0200,-9.5700){\makebox(0,0)[lb]{12}}%
% STR 2 0 3 0
% 3 1408 1120 1408 1216 2 0
% 13
\put(14.0800,-12.1600){\makebox(0,0)[lb]{13}}%
% STR 2 0 3 0
% 3 1408 1379 1408 1475 2 0
% 14
\put(14.0800,-14.7500){\makebox(0,0)[lb]{14}}%
% STR 2 0 3 0
% 3 1850 1178 1850 1274 2 0
% 15
\put(18.5000,-12.7400){\makebox(0,0)[lb]{15}}%
% STR 2 0 3 0
% 3 1821 1677 1821 1773 2 0
% 16
\put(18.2100,-17.7300){\makebox(0,0)[lb]{16}}%
% STR 2 0 3 0
% 3 2262 1062 2262 1158 2 0
% 17
\put(22.6200,-11.5800){\makebox(0,0)[lb]{17}}%
% STR 2 0 3 0
% 3 2262 1456 2262 1552 2 0
% 18
\put(22.6200,-15.5200){\makebox(0,0)[lb]{18}}%
% STR 2 0 3 0
% 3 294 1395 294 1456 2 0
% 1
\put(2.9400,-14.5600){\makebox(0,0)[lb]{1}}%
% STR 2 0 3 0
% 3 4922 1740 4922 1802 2 0
% 3
\put(49.2200,-18.0200){\makebox(0,0)[lb]{3}}%
% STR 2 0 3 0
% 3 3278 885 3278 981 2 0
% 4
\put(32.7800,-9.8100){\makebox(0,0)[lb]{4}}%
% STR 2 0 3 0
% 3 3270 1619 3270 1715 2 0
% 5
\put(32.7000,-17.1500){\makebox(0,0)[lb]{5}}%
% STR 2 0 3 0
% 3 3770 1629 3770 1725 2 0
% 9
\put(37.7000,-17.2500){\makebox(0,0)[lb]{9}}%
% STR 2 0 3 0
% 3 4595 1648 4595 1744 2 0
% 16
\put(45.9500,-17.4400){\makebox(0,0)[lb]{16}}%
% STR 2 0 3 0
% 3 5018 1062 5018 1158 2 0
% 17
\put(50.1800,-11.5800){\makebox(0,0)[lb]{17}}%
% STR 2 0 3 0
% 3 5008 1485 5008 1581 2 0
% 18
\put(50.0800,-15.8100){\makebox(0,0)[lb]{18}}%
% STR 2 0 3 0
% 3 3040 1395 3040 1456 2 0
% 1
\put(30.4000,-14.5600){\makebox(0,0)[lb]{1}}%
% STR 2 0 3 0
% 3 4912 819 4912 880 2 0
% 2
\put(49.1200,-8.8000){\makebox(0,0)[lb]{2}}%
% STR 2 0 3 0
% 3 3779 909 3779 1005 2 0
% 7
\put(37.7900,-10.0500){\makebox(0,0)[lb]{7}}%
% STR 2 0 3 0
% 3 3933 986 3933 1082 2 0
% 10
\put(39.3300,-10.8200){\makebox(0,0)[lb]{10}}%
% STR 2 0 3 0
% 3 4586 861 4586 957 2 0
% 12
\put(45.8600,-9.5700){\makebox(0,0)[lb]{12}}%
% STR 2 0 3 0
% 3 3818 1254 3818 1350 2 0
% 8
\put(38.1800,-13.5000){\makebox(0,0)[lb]{8}}%
% STR 2 0 3 0
% 3 3462 1264 3462 1360 2 0
% 6
\put(34.6200,-13.6000){\makebox(0,0)[lb]{6}}%
% STR 2 0 3 0
% 3 3933 1533 3933 1629 2 0
% 11
\put(39.3300,-16.2900){\makebox(0,0)[lb]{11}}%
% STR 2 0 3 0
% 3 4259 1533 4259 1629 2 0
% 11
\put(42.5900,-16.2900){\makebox(0,0)[lb]{11}}%
% STR 2 0 3 0
% 3 4182 986 4182 1082 2 0
% 10
\put(41.8200,-10.8200){\makebox(0,0)[lb]{10}}%
% STR 2 0 3 0
% 3 4576 1178 4576 1274 2 0
% 15
\put(45.7600,-12.7400){\makebox(0,0)[lb]{15}}%
% STR 2 0 3 0
% 3 4173 1149 4173 1245 2 0
% 13
\put(41.7300,-12.4500){\makebox(0,0)[lb]{13}}%
% STR 2 0 3 0
% 3 4173 1360 4173 1456 2 0
% 14
\put(41.7300,-14.5600){\makebox(0,0)[lb]{14}}%
% STR 2 0 3 0
% 3 2764 1124 2764 1185 2 0
% \fbox{$\Gamma(\lambda)$}
\put(27.6400,-11.8500){\makebox(0,0)[lb]{\fbox{$\Gamma(\lambda)$}}}%
% STR 2 0 3 0
% 3 5248 1340 5248 1401 2 0
% \fbox{$\Gamma'(\lambda)$}
\put(52.4800,-14.0100){\makebox(0,0)[lb]{\fbox{$\Gamma'(\lambda)$}}}%
% STR 2 0 3 0
% 3 2776 2144 2776 2205 2 0
% \fbox{$\Gamma(\lambda)$}
\put(27.7600,-22.0500){\makebox(0,0)[lb]{\fbox{$\Gamma(\lambda)$}}}%
% STR 2 0 3 0
% 3 5236 2432 5236 2493 2 0
% \fbox{$\Gamma'(\lambda)_{\mu}$}
\put(52.3600,-24.9300){\makebox(0,0)[lb]{\fbox{$\Gamma'(\lambda)_{\mu}$}}}%
% STR 2 0 3 0
% 3 4912 2844 4912 2906 2 0
% 3
\put(49.1200,-29.0600){\makebox(0,0)[lb]{3}}%
% STR 2 0 3 0
% 3 3268 1989 3268 2085 2 0
% 4
\put(32.6800,-20.8500){\makebox(0,0)[lb]{4}}%
% STR 2 0 3 0
% 3 3261 2723 3261 2819 2 0
% 5
\put(32.6100,-28.1900){\makebox(0,0)[lb]{5}}%
% STR 2 0 3 0
% 3 3760 2733 3760 2829 2 0
% 9
\put(37.6000,-28.2900){\makebox(0,0)[lb]{9}}%
% STR 2 0 3 0
% 3 4586 2752 4586 2848 2 0
% 16
\put(45.8600,-28.4800){\makebox(0,0)[lb]{16}}%
% STR 2 0 3 0
% 3 5018 2166 5018 2262 2 0
% 17
\put(50.1800,-22.6200){\makebox(0,0)[lb]{17}}%
% STR 2 0 3 0
% 3 4998 2589 4998 2685 2 0
% 18
\put(49.9800,-26.8500){\makebox(0,0)[lb]{18}}%
% STR 2 0 3 0
% 3 3011 2499 3011 2560 2 0
% 1
\put(30.1100,-25.6000){\makebox(0,0)[lb]{1}}%
% STR 2 0 3 0
% 3 4902 1923 4902 1984 2 0
% 2
\put(49.0200,-19.8400){\makebox(0,0)[lb]{2}}%
% STR 2 0 3 0
% 3 3770 2013 3770 2109 2 0
% 7
\put(37.7000,-21.0900){\makebox(0,0)[lb]{7}}%
% STR 2 0 3 0
% 3 4576 1965 4576 2061 2 0
% 12
\put(45.7600,-20.6100){\makebox(0,0)[lb]{12}}%
% STR 2 0 3 0
% 3 3894 2627 3894 2723 2 0
% 11
\put(38.9400,-27.2300){\makebox(0,0)[lb]{11}}%
% STR 2 0 3 0
% 3 4240 2637 4240 2733 2 0
% 11
\put(42.4000,-27.3300){\makebox(0,0)[lb]{11}}%
% STR 2 0 3 0
% 3 3952 2080 3952 2176 2 0
% 10
\put(39.5200,-21.7600){\makebox(0,0)[lb]{10}}%
% STR 2 0 3 0
% 3 4202 2080 4202 2176 2 0
% 10
\put(42.0200,-21.7600){\makebox(0,0)[lb]{10}}%
% STR 2 0 3 0
% 3 4173 2253 4173 2349 2 0
% 13
\put(41.7300,-23.4900){\makebox(0,0)[lb]{13}}%
% STR 2 0 3 0
% 3 4173 2454 4173 2550 2 0
% 14
\put(41.7300,-25.5000){\makebox(0,0)[lb]{14}}%
% STR 2 0 3 0
% 3 4442 2291 4442 2387 2 0
% 15
\put(44.4200,-23.8700){\makebox(0,0)[lb]{15}}%
% STR 2 0 3 0
% 3 4778 2291 4778 2387 2 0
% 15
\put(47.7800,-23.8700){\makebox(0,0)[lb]{15}}%
\end{picture}%